\documentclass[11pt]{article}
\usepackage{amssymb}
\usepackage{amsmath}
\usepackage{mathrsfs}
\usepackage{graphics}
\usepackage{graphicx}
\usepackage{xcolor}
\usepackage{subfigure}
\usepackage[T1]{fontenc}
\usepackage{latexsym,amssymb,amsmath,amsfonts,amsthm}\usepackage{txfonts}
\newcommand{\be}{\begin{eqnarray}}
\newcommand{\ben}{\begin{eqnarray*}}
\newcommand{\en}{\end{eqnarray}}
\newcommand{\enn}{\end{eqnarray*}}

\newtheorem{theorem}{Theorem}[section]
\newtheorem{lemma}{Lemma}[section]
\newtheorem{prp}[theorem]{Proposition}
\newtheorem{thm}[theorem]{Theorem}

\newtheorem{dfn}{Definition}[section]

\begin{document}
\renewcommand{\theequation}{\arabic{section}.\arabic{equation}}
\begin{titlepage}
\title{\bf Large deviations for nematic liquid crystals driven by pure jump noise
}
\author{  Rangrang Zhang$^{1,}$\thanks{Corresponding author.}\ \ \ Guoli Zhou$^{2}$\\
{\small $^1$ Department of  Mathematics,
Beijing Institute of Technology, Beijing, 100081, P R China}\\
{\small $^2$ School of Statistics and Mathematics, Chongqing University, Chongqing, 400044, P R China}\\
(  {\sf rrzhang@amss.ac.cn}\ \ {\sf zhouguoli736@126.com})}
\date{}
\end{titlepage}
\maketitle

\noindent\textbf{Abstract}:
In this paper, we establish a large deviation principle for a stochastic evolution equation which describes the system governing the nematic liquid crystals driven by pure jump noise. The proof is based on the weak convergence approach.



\section{Introduction}
As we all know, the obvious states of matter are the solid, the liquid and the gaseous state. The liquid crystal is an intermediate state of a matter, in between the liquid and the crystalline solid, i.e. it must possess some typical properties of a liquid as well as  some crystalline properties. The nematic liquid crystal phase is characterized by long-range orientational order, i.e. the molecules have no positional order but tend to align along a preferred direction. Much of the interesting phenomenology of liquid crystals involves the geometry and dynamics of the preferred axis, which is defined by a vector $\theta$. This vector is called a director. Since the sign as well as the magnitude of the director has no physical significance, it is taken to be unity.

The concrete description of the physical relevance of liquid crystals can be referred to Chandrasekhar \cite{Ch}, Warner and Terentjev \cite{W-T}, Gennes and Prost \cite{G-P} and the references therein.
In the 1960's, Ericksen \cite{Er} and Leslie \cite{Le} demonstrated the hydrodynamic theory of liquid crystals. Moreover, they expanded the continuum theory which has been widely used by most researchers to design the dynamics of the nematic liquid crystals. Inspired by this theory, the most fundamental form of dynamical system representing the motion of nematic liquid crystals has been procured by Lin and Liu \cite{L-L}.

The addition of a stochastic noise to this model is fully natural as it represents external random perturbations or a lack of knowledge of certain physical parameters. More precisely, we
consider the following nematic liquid crystals driven by a pure jump noise in $\mathbb{O}_T:=(0,T]\times \mathbb{O}$, $\mathbb{O}\subset \mathbb{R}^d$, $d=$2 or 3,
\begin{eqnarray}\label{equ-0}
\left\{
  \begin{array}{ll}
   du+[(u\cdot \nabla)u-\mu \Delta u+\nabla p]dt+ \nabla \cdot (\nabla \theta\odot \nabla \theta)dt=\int_{\mathbb{X}}G(t,u(t-),v)\tilde{N}(dt,dv), &\\
    \nabla\cdot u=0, &\\
    d\theta+[(u\cdot \nabla)\theta-\gamma \Delta \theta(t)]dt+f(\theta(t))dt=0, &
  \end{array}
\right.
\end{eqnarray}
where $\tilde{N}$ is the compensated time homogeneous Poisson random measure and $G, f$ are measurable functions will be specified later. There are several recent works about the existence and uniqueness of pathwise weak solution of (\ref{equ-0}), i.e. strong in the probabilistic sense and weak in the PDE sense. In \cite{Au},
Brze\'{z}niak, Hausenblas and Razafimandimby studied the Ginzburg-Landau approximation of the nematic liquid crystals under the influence of fluctuating external forces. In that paper, they proved the existence and uniqueness of local maximal solution in both 2D and 3D cases using fixed point argument. Also they have proved the existence of global strong solution to the problem in the 2D case.
Brze\'{z}niak, Manna and Panda \cite{B-M-P} studied the nematic liquid crystals driven by pure jump noise in both 2D and 3D cases. They proved the global well-posedness of strong solution in the 2D case and established the existence of weak martingale solution of this model in the 3D case, respectively.

 The purpose of this paper is to prove a large deviations for the 2D nematic liquid crystals driven by a pure jump noise, which provides the exponential decay  of small probabilities associated  with
the corresponding stochastic dynamical systems with small noise.
The proof of the large deviations will be based on the weak convergence approach introduced in Budhiraja, Chen and Dupuis \cite{B-C-D} and Budhiraja, Dupuis and Maroulas \cite{B-D-M}. As an important part of the proof, we need to obtain global well-posedness of the so called skeleton equation.
For the uniqueness, we adopt the method introduced in \cite{B-M-P}.  For the existence, we first apply the Faedo-Galerkin approximation method to construct a sequence of  approximating equations as in \cite{B-M-P}. We then show that the family of the solutions of the approximating equations is compact in an appropriate space and that any limit of the approximating solutions gives rise to a solution of the skeleton equation.
To complete the proof of the large deviation principle, we also need to study the weak convergence of the  perturbations of the
system (\ref{equ-0}) in the random directions of the Cameron-Martin space of the driving Brownian motions.

This paper is organized as follows. The mathematical formulation of nematic liquid crystals flows is in Section 2. In Section 3, we recall a general criterion obtained in Budhiraja, Dupuis and Maroulas \cite{B-D-M} and state the main result. Section 4 is devoted to the study of the skeleton equations. The large deviations is proved in Section 5.

\section{The mathematical framework}

Let $T>0$ and $\mathbb{O}\subset \mathbb{R}^2$ be a bounded domain with smooth boundary $\partial \mathbb{O}$. Consider the following two-dimensional stochastic evolution equations in $\mathbb{O}_T:=(0,T]\times \mathbb{O}$ given by
\begin{eqnarray}\label{eq-1}
\left\{
  \begin{array}{ll}
   du+[(u\cdot \nabla)u-\mu \Delta u+\nabla p]dt+ \nabla \cdot (\nabla \theta\odot \nabla \theta)dt=\int_{\mathbb{X}}G(t,u(t-),v)\tilde{N}(dt,dv), &\\
    \nabla\cdot u=0, &\\
    d\theta+[(u\cdot \nabla)\theta-\gamma \Delta \theta(t)]dt+f(\theta(t))dt=0, &
  \end{array}
\right.
\end{eqnarray}
where the vector field $u=u(x,t)$ denotes the velocity of the fluid, $\theta=\theta(x,t)$ is the director field, $p$ denoting the scalar pressure. $\tilde{N}$ is the compensated time homogeneous Poisson random measure on a certain locally compact Polish space $(\mathbb{X}, \mathcal{B}(\mathbb{X}))$. $G$ and $f$ are measurable functions, which will be specified in subsection \ref{s-2}. The symbol $\nabla \theta\odot \nabla \theta$ is the $2\times 2$-matrix with the entries
\begin{eqnarray*}
[\nabla \theta\odot \nabla \theta]_{i,j}=\sum^2_{k=1}\partial_{x_i}\theta^{(k)}\partial_{x_j}\theta^{(k)},\quad i,j=1,2.
\end{eqnarray*}
Without loss of generality, we assume that
\[
\mu=\gamma=1.
\]
The boundary and initial conditions for (\ref{eq-1}) are
\begin{eqnarray*}\notag
&u=0\ {\rm{and}}\ \frac{\partial \theta}{\partial n}=0\quad {\rm{on}}\  \partial \mathbb{O},&\\
&(u(0),\theta(0))=(u_0,\theta_0),&
\end{eqnarray*}
where $n$ is the outward  unit normal vector at each point $x$ of $\mathbb{O}$.

\subsection{Functional spaces}
Denote by $\mathbb{N}, \mathbb{R}, \mathbb{R}^+, \mathbb{R}^d$ the set of positive integers, real numbers, positive real numbers and $d-$dimensional real vectors, respectively. For a topology space $\mathcal{E}$, denote the corresponding Borel $\sigma-$field by $\mathcal{B}(\mathcal{E})$. For a metric space $\mathbb{X}$, $C([0,T];\mathbb{X})$ stands for the space of continuous functions from $[0,T]$ into $\mathbb{X}$ and $\mathcal{D}([0,T];\mathbb{X})$ represents the space of right continuous functions with left limits from $[0,T]$ into $\mathbb{X}$. For a metric space $\mathbb{Y}$, denote by $M_b(\mathbb{Y}), C_b(\mathbb{Y})$ the space of real valued bounded $\mathcal{\mathbb{Y}}/ \mathcal{\mathbb{R}}-$measurable maps and real valued bounded continuous functions, respectively.

Now, we follow closely the framework of \cite{B-M-P}.
For any $p\in [1,\infty)$ and $k\in \mathbb{N}$, let $(L^p(\mathbb{O}),|\cdot|_{L^p})$ and $(W^{k,p}(\mathbb{O}),\|\cdot\|_{W^{k,p}})$ be Lebesgue and Sobolev space of $\mathbb{R}^2$-valued functions, respectively. For $p=2$, set $W^{k,2}=H^k$. For instance, $H^1(\mathbb{O})$ is the Sobolev space of all $u\in L^2(\mathbb{O})$, for which there exist weak derivatives $\frac{\partial u}{\partial x_i}\in L^2(\mathbb{O}), i=1,2.$ It's well-known that $H^1(\mathbb{O})$ is a Hilbert space with the scalar product given by
\[
(u,v)_{H^1}:=(u,v)_{L^2}+(\nabla u, \nabla v)_{L^2}, \quad u,v \in H^1(\mathbb{O}).
\]
Now, define working spaces for the system (\ref{eq-1}) as
 \begin{eqnarray}\notag
 \mathcal{V}:=\left\{u\in C^{\infty}_c(\mathbb{O});\ \nabla\cdot u=0\right\},
 \end{eqnarray}
$\mathbb{H}$:= the closure of $ \mathcal{V}$ in $L^2(\mathbb{O})$ and
$\mathbb{V}$:= the closure of $ \mathcal{V}$ in $H^1(\mathbb{O})$.

In the space  $\mathbb{H}$, we equip it with the scalar product and the norm inherited from $L^2(\mathbb{O})$ and denote them by $(\cdot,\cdot)_{\mathbb{H}}$ and $|\cdot|_{\mathbb{H}}$, respectively, i.e.,
\[
(u,v)_{\mathbb{H}}:=(u,v)_{L^2}, \quad |u|_{\mathbb{H}}:=|u|_{L^2}:=|u|, \quad u,v\in \mathbb{H}.
\]
In the space $\mathbb{V}$, we equip it with the scalar product inherited from the Sobolev space $H^1(\mathbb{O})$, i.e.,
\[
(u,v)_{\mathbb{V}}:=(u,v)_{L^2}+((u,v)),
\]
where
\begin{eqnarray}\label{eq-6}
((u,v)):=(\nabla u,\nabla v)_{L^2}=\int_{\mathbb{O}}\frac{\partial u}{\partial x_1}\cdot\frac{\partial v}{\partial x_1}dx+\int_{\mathbb{O}}\frac{\partial u}{\partial x_2}\cdot \frac{\partial v}{\partial x_2}dx, \quad u,v\in \mathbb{V}.
\end{eqnarray}
The norm of $\mathbb{V}$ is defined as
\[
\|u\|^2_{\mathbb{V}}:=|u|^2_{\mathbb{H}}+\|u\|^2,
\]
where $\|u\|^2:=|\nabla u|^2$.

As we are working on a bounded domain, it's clear that
\[
\mathbb{V}\hookrightarrow \mathbb{H}\cong \mathbb{H}'\hookrightarrow \mathbb{V}',
\]
where the embedding is compact continuous.
Also, we have the embedding
\[
H^2\hookrightarrow H^1\hookrightarrow L^2\cong L^2 \hookrightarrow (H^1)'\hookrightarrow (H^2)'.
\]
\subsection{Some functionals}
Set
\begin{eqnarray}\label{eq-5}
A_1u:=((u, \cdot)), \quad u\in \mathbb{V},
\end{eqnarray}
where $((\cdot,\cdot))$ is defined by (\ref{eq-6}). If $u\in \mathbb{V}$, then $A_1u\in \mathbb{V}'$. By the Cauchy-Schwarz inequality, we deduce that
\begin{eqnarray}\label{eq-7}
|A_1u|_{\mathbb{V}'}\leq \|u\|, \quad u\in \mathbb{V}.
\end{eqnarray}
It's well-known that $A_1$ is a positive self-adjoint operator. Let $\{\varrho_i\}^{\infty}_{i=1}$ be the orthonormal basis of $\mathbb{H}$ composed of eigenfunctions of the Stokes operator $A_1$ with corresponding eigenvalues $0\leq \lambda_1\leq \lambda_2\leq\cdot\cdot\cdot\rightarrow \infty$ $ (A_1\varrho_i=\lambda_i \varrho_i)$. We will  use fractional powers of the operator $A_1$, denoted by $A^{\alpha}_1$, as well as their domains $D(A^{\alpha}_1)$ for $\alpha\in \mathbb{R}$. Note that
\[
D(A^{\alpha}_1)=\{u=\sum^{\infty}_{i=1}u_i\cdot \varrho_i:\sum^{\infty}_{i=1}\lambda_i^{2\alpha}u^2_i<\infty\}.
\]
We may endow $D(A^{\alpha}_1)$ with the inner product
\[
(u,v)_{D(A^{\alpha}_1)}=(A^{\alpha}_1 u,A^{\alpha}_1 v)_{\mathbb{H}}.
\]
Hence, $(D(A^{\alpha}_1), (\cdot,\cdot)_{D(A^{\alpha}_1)})$ is a Hilbert space and $\{\lambda^{-\alpha}_i \varrho_i\}_{i\in\mathbb{N}}$ is a complete orthonormal system of $D(A^{\alpha}_1)$. By Riesz representative theorem, $D(A^{-\alpha}_1)$ is the dual space of $D(A^{\alpha}_1)$.

Define a self-adjoint operator $A_2: H^1\rightarrow (H^1)'$ by
\begin{eqnarray}\label{eq-8}
\langle A_2\theta, w\rangle:=((\theta, w)):=\int_{\mathbb{O}}\nabla \theta\cdot \nabla wdx, \quad \theta, w\in H^1.
\end{eqnarray}
Let $\{\varsigma_i\}^{\infty}_{i=1}$ be the orthonormal basis of $L^2$ composed of eigenfunctions of the Stokes operator $A_2$.
 We  have
\begin{eqnarray}\label{eq-9}
\|A_2 \theta\|_{(H^1)'}\leq \|\theta\|_{H^1}.
\end{eqnarray}

Consider the following trilinear form (see \cite{Temam-1})
\[
b(u,v,w)=\sum^2_{i,j=1}\int_{\mathbb{O}}u^{(i)}\partial_{x_i}v^{(j)}w^{(j)}dx, \quad u\in L^p,\ v\in W^{1,q},\ w\in L^r,
\]
where $p,q,r\in [1,\infty]$ satisfying
\begin{eqnarray}\label{eq-2}
\frac{1}{p}+\frac{1}{q}+\frac{1}{r}\leq 1.
\end{eqnarray}
Referring to \cite{A}, by the Sobolev embedding Theorem and H\"{o}lder inequality, we obtain
\begin{eqnarray}
|b(u,v,w)|\leq C\|u\|_{\mathbb{V}}\|v\|_{\mathbb{V}}\|w\|_{\mathbb{V}},\quad u,v,w\in \mathbb{V},
\end{eqnarray}
for some positive constant $C$. Thus, $b$ is a continuous on $\mathbb{V}$.

Now, define a bilinear map $B: V\times V\rightarrow V'$ by
\begin{eqnarray*}
\langle B(u,v),w\rangle:=b(u,v,w)=\sum^2_{i,j=1}\int_{\mathbb{O}}u^{(i)}\partial_{x_i}v^{(j)}w^{(j)}dx.
\end{eqnarray*}
Then, referring to \cite{Temam-1}, it gives
\begin{lemma}\label{lem-1}
For any $u\in V$, $v\in V$, $w\in V$,
\begin{description}
  \item[(1)] $\langle B(u,v),w\rangle=-\langle B(u,w),v\rangle,\quad \langle B(u,v),v\rangle=b(u,v,v)=0.$
  \item[(2)] $\|B(u,v)\|_{\mathbb{V}'}\leq C|u|^{\frac{1}{2}}\|u\|^{\frac{1}{2}}|v|^{\frac{1}{2}}\|v\|^{\frac{1}{2}},$ \quad {\rm{for \ some\ positive
      \ constnat\ C}}.
\end{description}
\end{lemma}
Based on Lemma \ref{lem-1}, the operator $B$ can be uniquely extended to a bounded linear operator
\[
B:\mathbb{H}\times \mathbb{H}\rightarrow \mathbb{V}',
\]
 and it satisfies the following estimate
\begin{eqnarray} \label{eq-3}
 \|B(u,v)\|_{\mathbb{V}'}\leq C|u||v|.
\end{eqnarray}
For the convenience for written, denote $B(u):=B(u,u)$. Note that $B: \mathbb{V}\rightarrow \mathbb{V}' $ is locally Lipschitz continuous.

Now, define a bilinear mapping $\tilde{B}: H^1\times H^1\rightarrow (H^1)'$ as
\[
\langle \tilde{B}(u,v), w\rangle=b(u,v,w)\quad u,v,w\in H^1.
\]
We still denote by $\tilde{B}(\cdot,\cdot)$ the restriction of $\tilde{B}(\cdot,\cdot)$  to $\mathbb{V}\times H^2$, which map continuously from $\mathbb{V}\times H^2$ into $L^2$. According to \cite{Temam-1}, we have
\begin{lemma}\label{lem-2}
For $u\in \mathbb{V}$, $\theta\in H^2$, we have
\begin{description}
  \item[(1)] $\langle \tilde{B}(u, \theta), \theta\rangle=b(u, \theta,\theta)=0.$
  \item[(2)] $|\tilde{B}(u, \theta)|\leq C|u|^{\frac{1}{2}}\|u\|^{\frac{1}{2}}\|\theta\|^{\frac{1}{2}}|\Delta \theta|^{\frac{1}{2}}$,\quad {\rm{for\ some\ positive\ constant}} $C$.
\end{description}
\end{lemma}

Consider the trilinear form defined by
\[
m(\theta_1, \theta_2, u)=-\sum^2_{i,j,k=1}\int_{\mathbb{O}}\partial_{x_i} \theta^{(k)}_1\partial_{x_j} \theta^{(k)}_2\partial_{x_j} u^{(i)} dx
\]
for any $\theta_1\in W^{1,p}, \theta_2\in W^{1,q}$ and $u\in W^{1,r}$ with $p,q,r\in (1,\infty)$ satisfying  (\ref{eq-2}).

Define a bilinear operator $M: H^2\times H^2\rightarrow \mathbb{V}'$ such that for any $\theta_1, \theta_2\in H^2$
\[
\langle M(\theta_1, \theta_2), u \rangle=m(\theta_1, \theta_2,u), \quad u\in \mathbb{V}.
\]
Then, by H\"{o}lder inequality and Sobolev interpolation inequality, we have
\begin{eqnarray} \label{eq-4}
 \|M(\theta_1, \theta_2)\|_{\mathbb{V}'}\leq C\|\theta_1\|^{\frac{1}{2}}|\Delta\theta_1|^{\frac{1}{2}}\|\theta_2\|^{\frac{1}{2}}|\Delta\theta_2|^{\frac{1}{2}}.
\end{eqnarray}
For simplicity, we denote $M(\theta):=M(\theta,\theta)$.

Collecting all the above functionals, (\ref{eq-1}) can be written as
\begin{eqnarray}\label{equ-1}
\left\{
  \begin{array}{ll}
    du(t)+[A_1 u(t)+B(u(t))+M(\theta(t))]dt=\int_{\mathbb{O}}G(t,u(t),v)\tilde{N}(dt,dv), \\
   d\theta(t)+[A_2 \theta(t)+\tilde{B}(u(t),\theta(t))+f(\theta(t))]dt=0.
  \end{array}
\right.
\end{eqnarray}

\subsection{Hypotheses}\label{s-2}
To obtain the global well-posedness of (\ref{equ-1}), we introduce the following hypotheses stated in \cite{B-M-P}.
\begin{description}
   \item[\textbf{Hypothesis H0}] \begin{description}
                                   \item[(A)] $\tilde{N}$ is a compensated time homogeneous Poisson random measure on a locally compact space $(\mathbb{X},\mathcal{B}(\mathbb{X}))$ over a probability space $(\Omega, \mathcal{F}, \mathbb{P})$ with a $\sigma-$finite intensity measure $\vartheta$.
                                   \item[(B)] Let $G:[0,T]\times \mathbb{H}\times \mathbb{X}\rightarrow \mathbb{H}$ is a measurable function and there exists a constant $L$ such that
                                       \begin{eqnarray}
                                       \int_{\mathbb{X}}|G(t,u_1,v)-G(t,u_2,v)|^2_{\mathbb{H}}\vartheta(dv)\leq L|u_1-u_2|^2,\quad u_1,u_2\in \mathbb{H}, \ t\in [0,T].
                                       \end{eqnarray}
                                       and for each $p\geq 1$, there exists a constant $C_p$ such that
                                        \begin{eqnarray}
                                       \int_{\mathbb{X}}|G(t,u,v)|^p_{\mathbb{H}}\vartheta(dv)\leq C_p(1+|u|^p),\quad u\in \mathbb{H}, \ t\in [0,T].
                                       \end{eqnarray}
                                   \item[(C)] For $N\in \mathbb{N}$, numbers $b_j, j=0,\cdot\cdot\cdot, N$ with $b_j>0$, we define a function $\tilde{f}:[0,\infty)\rightarrow \mathbb{R}$ by
                                       \[
                                       \tilde{f}(r)=\sum^N_{j=0}b_j r^j,\quad {\rm{for\ any\ }} r\in \mathbb{R}_{+}.
                                       \]
                                       Define a map $f: \mathbb{R}^2\rightarrow \mathbb{R}^2$ by
\begin{eqnarray}
     f(\theta)=\tilde{f}(|\theta|^2)\theta.
\end{eqnarray}

Let $F: \mathbb{R}^2\rightarrow \mathbb{R} $ be a Frech\'{e}t differentiable map such that for any $\theta\in \mathbb{R}^2$ and $g\in \mathbb{R}^2$
\begin{eqnarray}\label{eq-38}
F'(\theta)[g]=f(\theta)\cdot g.
\end{eqnarray}
                                 \end{description}
\end{description}
Under \textbf{Hypothesis H0 (C)}, referring to Appendix D in \cite{B-M-P}, we have
\begin{lemma}\label{lem-3}

For any $\kappa_1>0$ and $\kappa_2>0$, there exists $C(\kappa_1)>0, C_1(\kappa_2)>0, C_2(\kappa_2)>0$ such that
\begin{eqnarray}\label{eq-10}
|\langle f(\theta_1)-f(\theta_2),\theta_1-\theta_2 \rangle|\leq \kappa_1|\nabla \theta_1-\nabla \theta_2|^2+C(\kappa_1)|\theta_1-\theta_2|^2\beta(\theta_1,\theta_2),
\end{eqnarray}
and
\begin{eqnarray}\notag
&&|\langle f(\theta_1)-f(\theta_2),\Delta(\theta_1-\theta_2) \rangle|\\
\label{eq-11}
&\leq& \kappa_2|\Delta \theta_1-\Delta \theta_2|^2+[C_1(\kappa_2)|\nabla(\theta_1-\theta_2)|^2+C_2(\kappa_2)|\theta_1-\theta_2|^2]\beta(\theta_1,\theta_2),
\end{eqnarray}
where
\begin{eqnarray}\label{eq-12}
\beta(\theta_1,\theta_2):=C(1+|\theta_1|^{2N}_{L^{4N+2}}+|\theta_2|^{2N}_{L^{4N+2}})^2.
\end{eqnarray}
Moreover, for any $\theta\in H^1$, it gives
\begin{eqnarray}\label{eq-13}
|f(\theta)|^2&\leq& C(1+|\theta|^q_{L^q}),\quad q=4N+2.
\end{eqnarray}
\end{lemma}
Now, we recall the definition of a strong solution to (\ref{equ-1}) in \cite{B-M-P}.
\begin{dfn}\label{dfn-1}
The system (\ref{equ-1}) has a strong solution if for every stochastic basis $(\Omega,\mathcal{F}, \{\mathcal{F}_t\}_{t\geq 0},\mathbb{P})$ and a time homogeneous Poisson random measure $\tilde{N}$ on $(\mathbb{X},\mathcal{B}(\mathbb{X}))$ over the stochastic basis with intensity measure $\vartheta$, there exist progressively measurable process $u:[0,T]\times \Omega\rightarrow \mathbb{H}$ with $\mathbb{P}-$a.e.
\begin{eqnarray}
u(\cdot, \omega)\in \mathcal{D}([0,T];\mathbb{H})\cap L^2([0,T]; \mathbb{V})
\end{eqnarray}
and progressively measurable process $\theta:[0,T]\times \Omega\rightarrow H^1 $ with $\mathbb{P}-$a.e.
\begin{eqnarray}
\theta(\cdot, \omega)\in C([0,T];H^1)\cap L^2([0,T]; H^2)
\end{eqnarray}
such that for all $t\in [0,T]$ and $\chi\in \mathbb{V}$, the following identity holds  $\mathbb{P}-$a.e.
\begin{eqnarray}\notag
&&(u(t),\chi)+\int^t_0\langle A_1 u(s), \chi\rangle ds+\int^t_0\langle B( u(s)), \chi\rangle ds+\int^t_0\langle M( \theta(s)), \chi\rangle ds\\ \label{eq-14}
&=&(u_0,\chi)+\int^t_0\int_{\mathbb{X}}(G(s,u(s),v),\chi)\tilde{N}(dsdv),
\end{eqnarray}
and for all $\upsilon\in H^1$, the following identity holds  $\mathbb{P}-$a.e.
\begin{eqnarray}\label{eq-15}
(\theta(t),\upsilon)+\int^t_0\langle A_2 \theta(s), \upsilon\rangle ds+\int^t_0\langle \tilde{B}( u(s),\theta(s)), \upsilon\rangle ds+\int^t_0\langle f( \theta(s)), \upsilon\rangle ds=(\theta_0,\upsilon).
\end{eqnarray}

\end{dfn}
According to \cite{B-M-P}, we have
\begin{thm}\label{thm-1}
Let the initial value $(u_0, \theta_0)\in \mathbb{H}\times H^1$. Under \textbf{Hypothesis H0}, the system (\ref{equ-1}) has a strong solution $(u,\theta)$ in the sense of   Definition \ref{dfn-1}. Also, the solution satisfies the following estimates
\begin{eqnarray}\label{eq-16}
\mathbb{E}\left[\sup_{t\in [0,T]}|u(t)|^2_{\mathbb{H}}+\int^T_0 \|u(t)\|^2_{\mathbb{V}}dt\right]<\infty,
\end{eqnarray}
and
\begin{eqnarray}\label{eq-17}
\mathbb{E}\left[\sup_{t\in [0,T]}\|\theta(t)\|^2_{H^1}+\int^T_0 \|\theta(t)\|^2_{H^2}dt\right]<\infty.
\end{eqnarray}

\end{thm}

\section{Preliminaries to large deviations}
In this section, we will recall a general criterion for a large deviation principle introduced by Budhiraja, Dupuis and Maroulas in \cite{B-D-M}. To this end, we closely follow the framework and notations in Budhiraja, Chen and Dupuis \cite{B-C-D} and Budhiraja, Dupuis and Maroulas \cite{B-D-M}.

Let $\{X^\varepsilon\}$ be a family of random variables defined on a probability space $(\Omega, \mathcal{F}, \mathbb{P})$ taking values in some Polish space $\mathcal{E}$. The large deviation principle is concerned with exponential decay of $\mathbb{P}(X^\varepsilon\in \cdot)$, as $\varepsilon\rightarrow 0$.

\begin{dfn}
(Rate Function) A function $I: \mathcal{E}\rightarrow [0,\infty]$ is called a rate function if for each $M<\infty$, the level set $\{x\in \mathcal{E}: I(x)\leq M\}$ is a compact subset of $\mathcal{E}$.  For $O\in \mathcal{B}(\mathcal{E})$, we define $I(O):=\inf_{x\in O}I(x)$.
\end{dfn}
\begin{dfn}
(Large Deviation Principle) The sequence $\{X^\varepsilon\}$ is said to satisfy a large deviation principle with rate function $I$ if the following two conditions hold.
\begin{description}
  \item[(a)] Large deviation upper bound. For each closed subset $F$ of $\mathcal{E}$,
  \[
  \lim \sup_{\varepsilon\rightarrow 0}\varepsilon \log \mathbb{P}(X^\varepsilon\in F)\leq -I(F),
  \]
  \item[(b)] Large deviation lower bound. For each open subset $G$ of $\mathcal{E}$,
  \[
  \lim \inf_{\varepsilon\rightarrow 0}\varepsilon \log \mathbb{P}(X^\varepsilon\in G)\geq -I(G).
  \]
\end{description}
\end{dfn}

\subsection{Controlled Poisson random measure}\label{s-1}
The following notations will be used. Let $\mathbb{X}$ be a locally compact Polish space. Set $C_c(\mathbb{X})$ be the space of continuous functions with compact supports. Denote
\[
\mathcal{M}_{FC}(\mathbb{X}):=\Big\{{\rm{measure}}\ \vartheta\  {\rm{on}}\ (\mathbb{X}, \mathcal{B}(\mathbb{X}))\ {\rm{such\ that}}\ \vartheta(K)<\infty\  {\rm{for\ every\ compact}}\ K\ {\rm{in}}\ \mathbb{X}\Big\}.
\]
Endow $\mathcal{M}_{FC}(\mathbb{X})$ with the weakest topology such that for every $f\in C_c(\mathbb{X})$, the function $\vartheta\rightarrow \langle f, \vartheta\rangle=\int_{\mathbb{X}}f(u)d\vartheta(u), \vartheta\in \mathcal{M}_{FC}(\mathbb{X})$ is continuous. This topology can be metrized such that $\mathcal{M}_{FC}(\mathbb{X})$ is a Polish space (see \cite{B-D-M}).

Let $T>0$, set $\mathbb{X}_T=[0,T]\times \mathbb{X}$. Fix a measure $\vartheta\in \mathcal{M}_{FC}(\mathbb{X})$ and let $\vartheta_T=\lambda_T\otimes \vartheta$, where $\lambda_T$ is Lebesgue measure on $[0,T]$. We recall the definition of Poisson random measure from \cite{I-W} that
\begin{dfn}
We call measure $\mathbf{n}$ a Poisson random measure on $\mathbb{X}_T$ with intensity measure $\vartheta_T$ is a $\mathcal{M}_{FC}(\mathbb{X})-$valued random variable such that
 \begin{description}
   \item[(1)] for each $B\in \mathcal{B}(\mathbb{X}_T)$ with $\vartheta_T(B)<\infty$, $\mathbf{n}(B)$ is a Poisson distribution with mean $\vartheta_T(B)$,
   \item[(2)] for disjoint $B_1,\cdot\cdot\cdot, B_k\in \mathcal{B}(\mathbb{X}_T)$, $\mathbf{n}(B_1),\cdot\cdot\cdot, \mathbf{n}(B_k) $ are mutually independent random variables.
 \end{description}

\end{dfn}
Denote by $\mathbb{P}$ the measure induced by $\mathbf{n}$ on $(\mathcal{M}_{FC}(\mathbb{X}_T),\mathcal{B}(\mathcal{M}_{FC}(\mathbb{X}_T)))$. Let $\mathbb{M}=\mathcal{M}_{FC}(\mathbb{X}_T)$. $\mathbb{P}$ is the unique probability measure on $(\mathbb{M}, \mathcal{B}(\mathbb{M}))$, under which the canonical map $N:\mathbb{M}\rightarrow \mathbb{M}$, $N(m):=m$ is a Poisson random measure with intensity measure $\vartheta_T$. In this paper, we also consider probability $\mathbb{P}_{\theta}$, for $\theta>0$, under which $N$ is a Poisson random measure with intensity $\theta \vartheta_T$. The corresponding expectation operators will be denoted by $\mathbb{E}$ and  $\mathbb{E}_{\theta}$, respectively.

Set
\[
\mathbb{Y}=\mathbb{X}\times[0,\infty), \quad \mathbb{Y}_T=[0,T]\times\mathbb{Y}.
\]
 Similarly, let $\bar{\mathbb{M}}=\mathcal{M}_{FC}(\mathbb{Y}_T)$ and let $\bar{\mathbb{P}}$ be the unique probability measure on $(\bar{\mathbb{M}},\mathcal{B}(\bar{\mathbb{M}}))$ under which the canonical mapping $\bar{N}:\bar{\mathbb{M}}\rightarrow \bar{\mathbb{M}}, \bar{N}(m):= m$ is a Poisson random measure with intensity measure $\bar{\vartheta}_T=\lambda_T\otimes\vartheta\otimes\lambda_{\infty}$, with $\lambda_{\infty}$ being Lebesgue measure on $[0,\infty)$. The expectation operator will be denoted by $\bar{\mathbb{E}}$. Let $\mathcal{F}_t:=\sigma\{\bar{N}((0,s]\times {O}):0\leq s\leq t, {O}\in \mathcal{B}(\mathbb{Y})\}$,
and denote by $\bar{\mathcal{F}}_t$ the completion under $\bar{\mathbb{P}}$. Let
 \[
 \bar{\mathcal{P}}\  {\rm{ be\ the\ predictable \ }} \ {\rm{\sigma-field\ on}}\ [0,T]\times \bar{\mathbb{M}}\ {\rm{with\ the\ filtration}}\  \{\bar{\mathcal{F}}_t:0\leq t\leq T\}\ {\rm{on }}\ (\bar{\mathbb{M}},\mathcal{B}(\bar{\mathbb{M}}))
 \]
 and
 \[
  \bar{\mathcal{A}}\ {\rm{ be\ the\ class\ of\ all \ (\bar{\mathcal{P}}\otimes \mathcal{B}(\mathbb{X}))/(\mathcal{B}[0,\infty))-measurable\ maps}}\  \varphi:\mathbb{X}_T\times \bar{\mathbb{M}}\rightarrow [0,\infty).
  \]
For $\varphi\in \bar{\mathcal{A}}$, define a counting process $N^\varphi$ on $\mathbb{X}_T$ by
\begin{eqnarray}
N^\varphi((0,t]\times U)=\int_{(0,t]\times U}\int_{(0,\infty)}I_{[0,\varphi(s,x)]}(r)\bar{N}(ds dxdr), \quad t\in [0,T], \quad U\in \mathcal{B}(\mathbb{X}).
\end{eqnarray}

$N^\varphi$ is the controlled random measure with $\varphi$ selecting the intensity for the points at location $x$ and time $s$, in a possibly random but nonanticipating way. If $\varphi(s,x, \bar{m})\equiv \theta\in(0,\infty)$. We write $N^\varphi=N^{\theta}$. Note that $N^{\theta}$ has the same distribution with respect to $\bar{\mathbb{P}}$ as $N$ has with respect to $\mathbb{P}_{\theta}$.
 Define $l:[0,\infty)\rightarrow[0,\infty)$ by
\[
l(r)=r \log r-r+1,\quad r\in [0, \infty).
\]
For any $\varphi\in \bar{\mathcal{A}}$, the quantity
\begin{eqnarray}
L_T(\varphi)=\int_{\mathbb{X}_T}l(\varphi(t,x,w))\vartheta_T(dtdx)
\end{eqnarray}
is well-defined as a $[0,\infty]-$valued random variable.

\subsection{A general criterion}
In order to state a general criteria for large deviation principle (LDP) obtained by Budhiraja et al. in \cite{B-D-M}, we introduce the following notations.
Define
\[
S^{M}=\{g:\mathbb{X}_T\rightarrow [0,\infty): L_T(g)\leq M\}, \quad S=\cup_{M\geq 1}S^M.
\]
A function $g\in S^{M}$ can be identified with a measure $\vartheta^g_T\in \mathbb{M}$, which is defined by
\[
\vartheta^g_T(O)=\int_{\mathbb{O}}g(s,x)\vartheta_T(dsdx),\quad O\in \mathcal{B}(\mathbb{X}_T).
\]

This identification induces a topology on $ S^{M}$ under which $ S^{M}$ is a compact space (see the Appendix of \cite{B-C-D}). Throughout this paper, we always  use this topology on $S^M$. Let
\[
\mathcal{U}^M=\{\varphi\in \bar{\mathcal{A}}: \varphi(\omega)\in S^M, \bar{\mathbb{P}}-a.e. \omega\},
\]
where $\bar{\mathcal{A}}$ is defined in subsection \ref{s-1}.

Let $\{\mathcal{G}^\varepsilon\}_{\varepsilon>0}$ be a family of measurable maps from $\bar{\mathbb{M}}$ to $\mathbb{U}$, where $\bar{\mathbb{M}}$ is introduced in subsection \ref{s-1} and $\mathbb{U}$ is a Polish space. Let $Z^\varepsilon=\mathcal{G}^\varepsilon(\varepsilon N^{\varepsilon^{-1}})$. Now, we list the following sufficient conditions for establishing LDP for the family $\{Z^\varepsilon\}_{\varepsilon>0}$.
\begin{description}
  \item[\textbf{Condition A} ] There exists a measurable map $\mathcal{G}^0: \bar{\mathbb{M}}\rightarrow \mathbb{U}$ such that the following hold.
\end{description}
\begin{description}
\item[(i)] For every $M<\infty$, let $g_n, g\in S^M$  be such that $g_n\rightarrow g$ as $n\rightarrow \infty$. Then,
      $\mathcal{G}^0(\vartheta^{g_n}_T)\rightarrow \mathcal{G}^0(\vartheta^{g}_T)$ in $\mathbb{U}$.
  \item[(ii)] For every $M<\infty$, let $\{\varphi_{\varepsilon}: \varepsilon>0\}\subset \mathcal{U}^M$ be such that $\varphi_{\varepsilon}$ converges in distribution to $\varphi$ as $\varepsilon\rightarrow 0$. Then,
     $\mathcal{G}^{\varepsilon}(\varepsilon N^{\varepsilon^{-1}\varphi_{\varepsilon}})$ converges to $\mathcal{G}^0(\vartheta^{\varphi}_T)$ in distribution.

\end{description}
The following result is due to Budhiraja et al. in \cite{B-D-M}.
\begin{thm}
Suppose the above \textbf{Condition A} holds. Then $Z^{\varepsilon}$ satisfies a large deviation principle on $\mathbb{U}$ with the good rate function $I$ given by
\begin{eqnarray}\label{eq-18}
I(f)&=&\inf_{\{g\in {S}: f=\mathcal{G}^0(\vartheta^g_T)\}}   \Big\{L_T(g)\Big\},\ \ \forall f\in\mathbb{U}.
\end{eqnarray}
By convention, $I(\emptyset)=\infty.$
\end{thm}

\subsection{Hypotheses and the statement of main results}
In order to obtain LDP for (\ref{equ-1}), we need additional conditions on the coefficients. Here, we adopt the same conditions as \cite{Y-Z-Z} and state some preliminary results from Budhiraja et al. \cite{B-C-D}.

Let $G:[0,T]\times \mathbb{H}\times \mathbb{X}\rightarrow H$ be a measurable mapping. Set
\begin{eqnarray*}
  |G(t,v)|_{0, \mathbb{H}}&:=&\sup_{u\in \mathbb{H}}\frac{|G(t,u,v)|_{\mathbb{H}}}{1+|u|_{\mathbb{H}}},\quad (t,v)\in [0,T]\times \mathbb{X},\\
  |G(t,v)|_{1, \mathbb{H}}&:=&\sup_{u_1,u_2\in \mathbb{H}, u_1\neq u_2}\frac{|G(t,u_1,v)-G(t,u_2,v)|_{\mathbb{H}}}{|u_1-u_2|_{\mathbb{H}}},\quad (t,v)\in [0,T]\times \mathbb{X},
  \end{eqnarray*}
\begin{description}
  \item[\textbf{Hypothesis H1} ] For $i=0,1$, there exists $\delta^i_1>0$ such that for all $E\in \mathcal{B}([0,T]\times \mathbb{X})$ satisfying $\vartheta_T(E)<\infty$, the following holds
      \[
      \int_E e^{\delta^i_1 |G(s,v)|^2_{i,\mathbb{H}}}\vartheta(dv)ds<\infty.
      \]
\end{description}
Now, we state the following Lemmas established by \cite{B-C-D} and \cite{Y-Z-Z}.
\begin{lemma}\label{lem-6}
Under \textbf{Hypothesis H0} and  \textbf{Hypothesis H1},
\begin{description}
  \item[(i)] For $i=0,1$ and every $M\in \mathbb{N}$,
  \begin{eqnarray}\label{eq-46}
  C^M_{i,1}&:=&\sup_{g\in S^M}\int_{\mathbb{X}_T}|G(s,v)|_{i,\mathbb{H}
  }|g(s,v)-1|\vartheta(dv)ds<\infty,\\
  \label{eq-47}
  C^M_{i,2}&:=&\sup_{g\in S^M}\int_{\mathbb{X}_T}|G(s,v)|^2_{i,\mathbb{H}
  }|g(s,v)+1|\vartheta(dv)ds<\infty.
  \end{eqnarray}
  \item[(ii)] For every $\eta>0$, there exists $\delta>0$ such that for any $A\subset [0,T]$ satisfying $\lambda_T(A)<\delta$
    \begin{eqnarray}\label{eq-48}
  \sup_{g\in S^M}\int_A\int_{\mathbb{X}}|G(s,v)|_{i,\mathbb{H}
  }|g(s,v)-1|\vartheta(dv)ds\leq \eta.
  \end{eqnarray}

\end{description}

\end{lemma}
\begin{lemma}\label{lem-7}
\begin{description}
  \item[(1)] For any $g\in S$, if $\sup_{t\in[0,T]}|Y(t)|_{\mathbb{H}}<\infty$, then
  \[
  \int_{\mathbb{X}}G(\cdot,Y(\cdot),v)(g(\cdot,v)-1)\vartheta(dv)\in L^1([0,T];\mathbb{H}).
  \]
  \item[(2)] If the family of mappings $\{Y_n:[0,T]\rightarrow H, n\geq 1\}$ satisfying $\sup_n\sup_{t\in[0,T]}|Y_n(t)|_{\mathbb{H}}<\infty$, then
      \[
      \tilde{C}_{M}:=\sup_{g\in S^M}\sup_n \int^T_0|\int_{\mathbb{X}}G(t,Y_n(t),v)(g(t,v)-1)\vartheta(dv)|_{\mathbb{H}}ds<\infty.
      \]
\end{description}
\end{lemma}

\begin{lemma}\label{lem-8}
Let $h:[0,T]\times \mathbb{X}\rightarrow \mathbb{R}$ be a measurable function such that
\[
\int_{\mathbb{X}_T}|h(s,v)|^2\vartheta(dv)ds<\infty,
\]
and for all $\delta\in (0,\infty)$ and $E\in \mathcal{B}([0,T]\times \mathbb{X})$ satisfying $\vartheta_T(E)<\infty$,
\[
\int_E \exp(\delta |h(s,v)|)\vartheta(dv)ds<\infty.
\]
Then, we have
\begin{description}
  \item[(1)] Fix $M\in \mathbb{N}$. Let $g_n, g\in S^M$ be such that $g_n\rightarrow g$ as $n\rightarrow \infty$. Then
      \[
      \lim_{n\rightarrow \infty}\int_{\mathbb{X}_T}h(s,v)(g_n(s,v)-1)\vartheta(dv)ds=\int_{\mathbb{X}_T}h(s,v)(g(s,v)-1)\vartheta(dv)ds.
      \]
  \item[(2)] Fix $M\in \mathbb{N}$. Given $\varepsilon>0$, there exists a compact set $K_{\varepsilon}\subset \mathbb{X}$, such that
      \[
      \sup_{g\in S^M}\int^T_0\int_{K^c_{\varepsilon}}|g(s,v)-1|\vartheta(dv)ds\leq \varepsilon.
      \]
  \item[(3)] For every compact set $K\subset \mathbb{X}$,
  \[
  \lim_{M\rightarrow \infty}\sup_{g\in S^M}\int^T_0\int_{K}|h(s,v)|I_{\{h\geq M\}}g(s,v)\vartheta(dv)ds=0.
  \]
\end{description}
\end{lemma}

In this paper, we consider the following nematic liquid crystals driven by small multiplicative L\'{e}vy noise:
\begin{eqnarray}\label{eq-19}
\left\{
  \begin{array}{ll}
    du^\varepsilon(t)+[A_1 u^\varepsilon(t)+B(u^\varepsilon(t))+M(\theta^\varepsilon(t))]dt=\varepsilon\int_{\mathbb{X}}G(t,u^\varepsilon(t),v)\tilde{N}^{\varepsilon^{-1}}(dtdv), \\
   d\theta^\varepsilon(t)+[A_2 \theta^\varepsilon(t)+\tilde{B}(u^\varepsilon(t),\theta^\varepsilon(t))+f(\theta^\varepsilon(t))]dt=0.
  \end{array}
\right.
\end{eqnarray}
By Theorem \ref{thm-1}, under \textbf{Hypothesis H0}, there exists a unique strong solution of (\ref{equ-1}) in $\mathcal{D}([0,T];\mathbb{H})\times C([0,T];H^1)$. Therefore, there exists a Borel-measurable mapping:
\[
\mathcal{G}^{\varepsilon}: \bar{\mathbb{M}}\rightarrow \mathcal{D}([0,T];\mathbb{H}) \times C([0,T];H^1)
\]
such that $(u^{\varepsilon}(\cdot), \theta^{\varepsilon}(\cdot))=\mathcal{G}^{\varepsilon}(\varepsilon {N}^{\varepsilon^{-1}})$.

For $g\in S$, consider the following skeleton equations
\begin{eqnarray}\label{eq-20}
\left\{
  \begin{array}{ll}
    du^g(t)+[A_1 u^g(t)+B(u^g(t))+M(\theta^g(t))]dt=\int_{\mathbb{X}}G(t,u^g(t),v)(g(t,v)-1)\vartheta(dv)dt, \\
   d\theta^g(t)+[A_2 \theta^g(t)+\tilde{B}(u^g(t),\theta^g(t))+f(\theta^g(t))]dt=0.
  \end{array}
\right.
\end{eqnarray}
The solution $(u^g,\theta^g)$ defines a mapping $\mathcal{G}^0: \bar{\mathbb{M}}\rightarrow \mathcal{D}([0,T];\mathbb{H})\times C([0,T];H^1)$ such that  $(u^g(\cdot),\theta^g(\cdot))=\mathcal{G}^0(\vartheta^g_T)$.

In this paper, our main result is
\begin{thm}
Let $(u_0,\theta_0) \in \mathbb{H}\times H^1 $. Under \textbf{Hypothesis H0} and \textbf{Hypothesis H1}, $(u^{\varepsilon},\theta^{\varepsilon})$ satisfies a large deviation principle on $\mathcal{D}([0,T]; \mathbb{H})\times C([0,T]; H^1)$ with the good rate function $I$ defined by (\ref{eq-18}) with respect to the uniform convergence.
\end{thm}
\begin{proof}
According to Theorem \ref{thm-1}, we need to prove (i) and (ii) in \textbf{Condition A}. The verification of (i) will be established by Proposition \ref{prp-1}, (ii) will be proved by Theorem \ref{thm-3}.
\end{proof}

\section{The skeleton equation}
In this section, we will show that the skeleton equation (\ref{eq-20}) admits a unique solution for every $g\in S$.

Let $K$ be a Banach space with norm $\|\cdot\|_K$.
Given $p>1, \alpha\in (0,1)$, as in \cite{FG95}, let $W^{\alpha,p}([0,T]; K)$ be the Sobolev space of all $u\in L^p([0,T];K)$ such that
\[
\int^T_0\int^T_0\frac{\|u(t)-u(s)\|_K^{ p}}{|t-s|^{1+\alpha p}}dtds< \infty,
\]
endowed with the norm
\[
\|u\|^p_{W^{\alpha,p}([0,T]; K)}=\int^T_0\|u(t)\|_K^pdt+\int^T_0\int^T_0\frac{\|u(t)-u(s)\|_K^{ p}}{|t-s|^{1+\alpha p}}dtds.
\]
The following results can be found in \cite{FG95}.
\begin{lemma}\label{lem-5}
Let $B_0\subset B\subset B_1$ be Banach spaces, $B_0$ and $B_1$ reflexive, with compact embedding $B_0\subset B$. Let $p\in (1, \infty)$ and $\alpha \in (0, 1)$ be given. Let $X$ be the space
\[
X= L^p([0, T]; B_0)\cap W^{\alpha, p}([0,T]; B_1),
\]
endowed with the natural norm. Then the embedding of $X$ in $L^p([0,T];B)$ is compact.
\end{lemma}
\begin{lemma}\label{lem-4}
For $V$ and $H$ are two Hilbert spaces ($V'$ is the dual space of $V$) with $V\subset\subset H=H'\subset V'$, where $V\subset\subset H$ denotes $V$ is compactly embedded in $H$. If $u\in L^2([0,T];V)$, $\frac{du}{dt}\in L^2([0,T];V')$, then $u\in C([0,T];H)$.
\end{lemma}

For the skeleton equation (\ref{eq-20}), we have
\begin{thm}\label{thm-2}
Given $(u_0,\theta_0)\in \mathbb{H}\times H^1 $ and $g\in S$. Assume \textbf{Hypothesis H0} and \textbf{Hypothesis H1} hold, then there exists a unique solution $(u^g, \theta^g)$ such that
\[
u^g\in C([0,T];\mathbb{H})\cap L^2([0,T];\mathbb{V}),\quad \theta^g\in C([0,T];H^1)\cap L^2([0,T];H^2),
\]
and
\begin{eqnarray}\notag
u^g(t)&=&u_0-\int^t_0 A_1 u^g(s)ds-\int^t_0 B( u^g(s))ds-\int^t_0 M( \theta^g(s))ds\\
\label{eq-21}
&&\quad +\int^t_0\int_{\mathbb{X}}G(s,u^g(s),v)(g(s,v)-1)\vartheta(dv)ds,\\
\label{eq-22}
\theta^g(t)&=&\theta_0-\int^t_0 A_2 \theta^g(s)ds-\int^t_0 \tilde{B}( u^g(s),\theta^g(s))ds-\int^t_0 f (\theta^g(s))ds.
\end{eqnarray}
Moreover, for any $M\in \mathbb{N}$, there exists $C(p,M)>0$ such that
\begin{eqnarray}\label{eq-23}
\sup_{g\in S^M} \left(\sup_{s\in [0,T]}|\theta^g(s)|^p_{L^2}+\int^T_0|\theta^g(s)|^{p-2}_{L^2}(\|\theta^g(s)\|^2+|\theta^g(s)|^{2N+2}_{L^{2N+2}})ds\right)\leq C(p,M),
\end{eqnarray}
and
\begin{eqnarray}\label{eq-24}
\sup_{g\in S^M}\left( \sup_{s\in [0,T]}(\Psi(\theta^g(s))+|u^g(s)|^2)^p+\Big(\int^T_0(\|u^g(s)\|^2+|\Delta \theta^g(s)-f(\theta^g(s))|^2)ds\Big)^p\right)\leq C(p,M),
\end{eqnarray}
where $\Psi(\theta^g(s)):=\frac{1}{2}\|\theta^g(s))\|^2+\frac{1}{2}\int_{\mathbb{O}}G(|\theta^g(s)|^2)dx$.
\end{thm}
\begin{proof}
(Existence) \ We apply the Faedo-Galerkin approximation method to deduce the existence of solution of (\ref{equ-1}). Let $\Phi_n:\mathbb{R}\rightarrow [0,1]$ be a smooth function such that $\Phi_n(t)=1$, if $|t|\leq n$, $\Phi_n(t)=0$, if $|t|> n+1$. Define $\chi^1_n: \mathbb{H}\rightarrow \mathbb{H}$ as $\chi^1_n(u)=\Phi_n(|u|_{\mathbb{H}})u$, and $\chi^2_n: L^2\rightarrow L^2$ as $\chi^2_n(\theta)=\Phi_n(|\theta|_{L^2})\theta$. Define the following finite dimensional spaces for any $n\in \mathbb{N}$,
\[
\mathbb{H}_n:={\rm{Span}}\{\varrho_1,\cdot\cdot\cdot,\varrho_n\},\quad \mathbb{L}_n:={\rm{Span}}\{\varsigma_1,\cdot\cdot\cdot,\varsigma_n\}.
\]
Denote by $P_n$ the projection from $\mathbb{H}$ onto $\mathbb{H}_n$, and $\tilde{P}_n$ the projection from $L^2$ onto $\mathbb{L}_n$. Let
\begin{eqnarray*}
B_n(u)&:= &\chi^1_n(u)B(u), \quad u\in \mathbb{H}_n,\\
M_n(\theta)&:= &\chi^2_n(\theta)M(\theta), \quad \theta\in \mathbb{L}_n,\\
\tilde{B}_n(\theta)&:= &\chi^1_n(u)\tilde{B}(u,\theta), \quad u\in \mathbb{H}_n, \ \theta\in \mathbb{L}_n.
\end{eqnarray*}
Based on the above mappings, consider the following Faedo-Galerkin approximations: $(u_n(t),\theta_n(t))\in \mathbb{H}_n\times \mathbb{L}_n$, which is the solution of
\begin{eqnarray}\label{eq-25}
&du_n(t)+A_1 u_n(t)dt+P_n B_n(u_n(t))dt+P_n M_n(\theta_n(t))dt=P_n\int_{\mathbb{X}}G(t, u_n(t),v)(g(t,v)-1)\vartheta(dv)dt,&\\
\label{eq-26}
&d\theta_n(t)+A_2 \theta_n(t)dt+\tilde{P}_n \tilde{B}_n(u_n(t),\theta_n(t))dt+\tilde{P}_n f_n(\theta_n(t))dt=0,&
\end{eqnarray}
with the initial condition $(u_n(0), \theta_n(0))=(P_nu_0, \tilde{P}_n \theta_0)$.

Since $B_n, M_n, \tilde{B}_n$ are all globally Lipschitz continuous, the existence of solutions to (\ref{eq-25})-(\ref{eq-26}) can be obtained using similar method as \cite{A-B-W}.

Now, for the solution $(u_n(t),\theta_n(t))$ of (\ref{eq-25})-(\ref{eq-26}), we aim to show that, for any $p\geq 1$,
\begin{eqnarray}\label{eq-27}
\sup_{n} \left(\sup_{s\in [0,T]}|\theta_n(s)|^p_{L^2}+\int^T_0|\theta_n(s)|^{p-2}_{L^2}(\|\theta_n(s)\|^2+|\theta_n(s)|^{2N+2}_{L^{2N+2}})ds\right)\leq C(p,M),
\end{eqnarray}
and
\begin{eqnarray}\label{eq-28}
\sup_{n}\left( \sup_{s\in [0,T]}(\Psi(\theta_n(s))+|u_n(s)|^2)^p+\Big(\int^T_0(\|u_n(s)\|^2+|\Delta \theta_n(s)-f(\theta_n(s))|^2)ds\Big)^p\right)\leq C(p,M),
\end{eqnarray}
where $\Psi(\theta_n(s)):=\frac{1}{2}\|\theta_n(s))\|^2+\frac{1}{2}\int_{\mathbb{O}}F(|\theta_n(s)|^2)dx$ and  $F$ is defined by (\ref{eq-38}).

Firstly, we prove (\ref{eq-27}). For $p\geq 2$, let $\psi(\cdot)$ be the mapping defined by
\begin{eqnarray}\label{eq-29}
\psi(\theta(t)):=\frac{1}{p}|\theta(t)|^p,\quad \theta\in L^2.
\end{eqnarray}
The first Fr\'{e}chet derivative is
\begin{eqnarray}\label{eq-30}
\psi'(\theta)[h]=|\theta(t)|^{p-2}\langle \theta,h\rangle.
\end{eqnarray}
Based on (\ref{eq-26}), (\ref{eq-29}) and (\ref{eq-30}), we deduce that
\begin{eqnarray*}\label{eq-31}
d\psi(\theta_n(t))=-|\theta_n(t)|^{p-2}\langle A_2\theta_n(t)+\tilde{B}_n(u_n(t),\theta_n(t))+f_n(\theta_n(t)),\theta_n(t)\rangle dt.
\end{eqnarray*}
Due to (\ref{eq-29}) and Lemma \ref{lem-2}, we arrive at
\begin{eqnarray}\label{eq-32}
|\theta_n(t)|^p+\int^t_0|\theta_n(s)|^{p-2}\|\theta_n(s)\|^2ds+\int^t_0|\theta_n(s)|^{p-2}\langle f_n(\theta_n(s)),\theta_n(s)\rangle ds\leq |\theta_0|^p.
\end{eqnarray}
Referring to equations (5.12) and (5.13) in \cite{B-M-P}, it gives that
\begin{eqnarray}\label{eq-33}
\langle f(\theta), \theta\rangle\geq\int_{\mathbb{O}}|\theta(x)|^{2N+2}dx-C\int_{\mathbb{O}}|\theta(x)|^2dx.
\end{eqnarray}
Putting (\ref{eq-33}) into (\ref{eq-32}), we deduce that
\begin{eqnarray*}\label{eq-34}
|\theta_n(t)|^p+\int^t_0|\theta_n(s)|^{p-2}\|\theta_n(s)\|^2ds+\int^t_0|\theta_n(s)|^{p-2}(|\theta_n(s)|^{2N+2}_{L^{2N+2}}-C|\theta_n(s)|^2) ds\leq |\theta_0|^p,
\end{eqnarray*}
which implies that
\begin{eqnarray}\label{eq-35}
|\theta_n(t)|^p+\int^t_0|\theta_n(s)|^{p-2}\|\theta_n(s)\|^2ds+\int^t_0|\theta_n(s)|^{p-2}|\theta_n(s)|^{2N+2}_{L^{2N+2}} ds\leq |\theta_0|^p+C\int^t_0|\theta_n(s)|^pds.
\end{eqnarray}
Applying Gronwall's inequality, we reach
\begin{eqnarray}\label{eq-36}
\sup_{t\in[0,T]}|\theta_n(t)|^p\leq |\theta_0|^pe^{CT}.
\end{eqnarray}
Combining (\ref{eq-35}) and (\ref{eq-36}), we conclude that
\begin{eqnarray*}\notag
&&\int^T_0|\theta_n(s)|^{p-2}\|\theta_n(s)\|^2ds+\int^T_0|\theta_n(s)|^{p-2}|\theta_n(s)|^{2N+2}_{L^{2N+2}} ds\\ \notag
&\leq& |\theta_0|^p+C\int^T_0|\theta_n(s)|^pds\\
\label{eq-37}
&\leq& |\theta_0|^p(1+CTe^{CT}).
\end{eqnarray*}
Thus, we complete the result (\ref{eq-27}).

For (\ref{eq-28}), we firstly define a stopping time
\begin{eqnarray}
\tau^R_n:=\inf\{t\geq 0:|u_n(t)|\geq R \ {\rm{or}}\ |\theta_n(t)|_{L^2}\geq R\ {\rm{or}}\ \|\theta_n(t)\|_{H^1}\geq R\}\wedge T.
\end{eqnarray}
From (\ref{eq-27}), we deduce that $\tau^R_n\uparrow T, \mathbb{P}-a.s.$, as $R\uparrow \infty$.

For any $t\in [0,\tau^R_n]$, define a mapping $\phi(\cdot)$ as
\[
\phi(u(t))=\frac{1}{2}|u(t)|^2,\quad u\in\mathbb{H}.
\]
Then, we have
\begin{eqnarray*}
d\phi(u_n(t))&=&-\langle A_1u_n(t)+B_n(u_n(t))+M_n(\theta_n(t)),u_n(t)\rangle dt\\
&&\quad +\langle P_n \int_{\mathbb{X}}G(t,u_n(t),v)(g(t,v)-1)\vartheta(dv),u_n(t)\rangle dt.
\end{eqnarray*}
Using Lemma \ref{lem-1}, we obtain
\begin{eqnarray*}\label{eq-39}
d\phi(u_n(t))+\|u_n(t)\|^2dt=-\langle M_n(\theta_n(t)),u_n(t)\rangle dt
 +\langle P_n \int_{\mathbb{X}}G(t,u_n(t),v)(g(t,v)-1)\vartheta(dv),u_n(t)\rangle dt.
\end{eqnarray*}
Let $\Psi(\cdot)$ be the mapping defined by
\[
\Psi(\theta)=\frac{1}{2}\|\theta\|^2+\frac{1}{2}\int_{\mathbb{O}}F(|\theta|^2)dx.
\]
Using (\ref{eq-38}), the first Fr\'{e}chet derivative is
\[
\Psi'(\theta)[g]=\langle \nabla \theta, \nabla g\rangle+\langle f(\theta), g\rangle=\langle -\Delta \theta+f(\theta), g\rangle.
\]
Then, we have
\begin{eqnarray*}\notag
d\Psi(\theta_n(t))&=&\Psi'(\theta_n)[d\theta_n(t)]=\langle -\Delta \theta_n+f_n(\theta_n), d\theta_n(t)\rangle\\
\notag
&=&\langle -\Delta \theta_n+f_n(\theta_n), -A_2 \theta_n(t)-f_n\theta_n(t)-\tilde{B}_n(u_n(t),\theta_n(t))\rangle dt\\
\label{eq-40}
&=&-|\Delta \theta_n-f_n(\theta_n)|^2dt-\langle\tilde{B}_n(u_n(t),\theta_n(t)), -\Delta \theta_n+f_n(\theta_n)\rangle dt.
\end{eqnarray*}
Referring to (5.28)-(5.29) in \cite{B-M-P}, it gives
\begin{eqnarray}\label{eq-41}
\langle\tilde{B}_n(u_n(t),\theta_n(t)), -\Delta \theta_n+f_n(\theta_n)\rangle=-\langle M_n(\theta_n),u_n\rangle.
\end{eqnarray}
This implies
\begin{eqnarray}\label{eq-42}
d\Psi(\theta_n(t))+|\Delta \theta_n-f_n(\theta_n)|^2dt=\langle M_n(\theta_n),u_n\rangle dt.
\end{eqnarray}
Adding (\ref{eq-39}) and (\ref{eq-42}), we get
\begin{eqnarray}\notag
&&d[\Psi(\theta_n(t))+\phi(u_n(t))]+(\|u_n(t)\|^2+|\Delta \theta_n-f_n(\theta_n)|^2)dt\\
\label{eq-43}
&=&\langle P_n \int_{\mathbb{X}}G(t,u_n(t),v)(g(t,v)-1)\vartheta(dv),u_n(t)\rangle dt.
\end{eqnarray}
Since
\begin{eqnarray*}
&&\langle P_n \int_{\mathbb{X}}G(t,u_n(t),v)(g(t,v)-1)\vartheta(dv),u_n(t)\rangle \\
&\leq& \int_{\mathbb{X}}|G(t,u_n(t),v)||g(t,v)-1||u_n(t)|\vartheta(dv)\\
&\leq& \int_{\mathbb{X}}\frac{|G(t,u_n(t),v)|}{1+|u_n(t)|}|g(t,v)-1|(1+|u_n(s)|)|u_n(t)|\vartheta(dv)\\
&\leq& \int_{\mathbb{X}}|G(t,v)|_{0,\mathbb{H}}|g(t,v)-1|(1+2|u_n(t)|^2)\vartheta(dv)\\
&\leq& \int_{\mathbb{X}}|G(t,v)|_{0,\mathbb{H}}|g(t,v)-1|\vartheta(dv)+2\int_{\mathbb{X}}|G(t,v)|_{0,\mathbb{H}}|g(t,v)-1||u_n(t)|^2\vartheta(dv),
\end{eqnarray*}
we conclude that
\begin{eqnarray}\notag
&&[\Psi(\theta_n(t))+|u_n(t)|^2]+\int^t_0(\|u_n(s)\|^2+|\Delta \theta_n-f_n(\theta_n)|^2)ds\\
\notag
&\leq& \Psi(\theta_0)+|u_0|^2+\int^t_0\int_{\mathbb{X}}|G(s,v)|_{0,\mathbb{H}}|g(s,v)-1|\vartheta(dv)ds\\
\label{eq-44}
&&\quad +2\int^t_0|u_n(s)|^2\int_{\mathbb{X}}|G(s,v)|_{0,\mathbb{H}}|g(s,v)-1|\vartheta(dv)ds.
\end{eqnarray}
Applying Gronwall inequality to (\ref{eq-44}), we have
\begin{eqnarray}\notag
&&\sup_{t\in[0,\tau^R_n]}[\Psi(\theta_n(t))+|u_n(t)|^2]+\int^{\tau^R_n}_0(\|u_n(s)\|^2+|\Delta \theta_n-f_n(\theta_n)|^2)ds\\
\notag
&\leq& \Big[\Psi(\theta_0)+|u_0|^2+\int^{\tau^R_n}_0\int_{\mathbb{X}}|G(s,v)|_{0,\mathbb{H}}|g(s,v)-1|\vartheta(dv)ds\Big]\\
\notag
&&\quad \times \exp\Big\{2\int^{\tau^R_n}_0\int_{\mathbb{X}}|G(s,v)|_{0,\mathbb{H}}|g(s,v)-1|\vartheta(dv)ds\Big\}\\ \notag
&\leq& \Big[\Psi(\theta_0)+|u_0|^2+\int^{T}_0\int_{\mathbb{X}}|G(s,v)|_{0,\mathbb{H}}|g(s,v)-1|\vartheta(dv)ds\Big]\\
\label{eq-45}
&&\quad \times \exp\Big\{2\int^{T}_0\int_{\mathbb{X}}|G(s,v)|_{0,\mathbb{H}}|g(s,v)-1|\vartheta(dv)ds\Big\}.
\end{eqnarray}
Utilizing (\ref{eq-46}), we deduce that
\begin{eqnarray}\notag
&&\sup_{t\in[0,\tau^R_n]}[\Psi(\theta_n(t))+|u_n(t)|^2]+\int^{\tau^R_n}_0(\|u_n(s)\|^2+|\Delta \theta_n-f_n(\theta_n)|^2)ds\\
\label{eq-52}
&\leq& [\Psi(\theta_0)+|u_0|^2+C^M_{0,1}]Te^{C^M_{0,1}T}.
\end{eqnarray}
As the constant in the right hand side of (\ref{eq-52}) is independent of $R$ and $n$, passing to the limit as $R\rightarrow \infty$, we obtain
\begin{eqnarray}\label{eq-52-1}
\sup_{s\in [0,T]}[\Psi(\theta_n(t))+|u_n(t)|^2]+\int^T_0(\|u_n(s)\|^2+|\Delta \theta_n-f_n(\theta_n)|^2)ds
\leq C(p,T).
\end{eqnarray}

Moreover, with the help of (\ref{eq-52-1}), we can obtain an estimate for $\Delta \theta_n$ and $\|\theta_n(t)\|^2_{H^1}$ using similar method as Proposition 5.6 in \cite{B-M-P}. Concretely, for any $p\geq 1$, there exists a positive constant $C$ independent of $n$ such that
\begin{eqnarray}\label{eq-49}
|\int^T_0|\Delta \theta_n(s) |^2ds|^p\leq C(p),\quad \sup_{t\in[0,T]}\|\theta_n(t)\|^{2p}_{H^1}\leq C(p).
\end{eqnarray}
In the following, we want to prove that for $\alpha\in (0,\frac{1}{2})$, there exists $C({\alpha}), L(\alpha)>0$ such that
\begin{eqnarray}\label{eq-50}
\sup_{n\geq 1}\|u_n\|^2_{W^{\alpha,2}([0,T];\mathbb{V}')}&\leq& C(\alpha).\\
\label{eq-51}
\sup_{n\geq 1}\|\theta_n\|^2_{W^{\alpha,2}([0,T];(H^2)')}&\leq& L(\alpha).
\end{eqnarray}
Firstly, $u_n(t)$ can be written as
\begin{eqnarray*}
u_n(t)&=&P_n u_0-\int^t_0 A_1 u_n(s)ds-\int^t_0B_n(u_n(s))ds-\int^t_0M_n(\theta_n(s))ds\\
&&\quad +\int^t_0\int_{\mathbb{X}}G(s, u_n(s),v)(g(s,v)-1)\vartheta(dv)ds\\
&:=&I^1_n+I^2_n(t)+I^3_n(t)+I^4_n(t)+I^5_n(t).
\end{eqnarray*}
Clearly, $|I^1_n|^2\leq C_1$. Since $\|A_1 u_n\|_{\mathbb{V}'}\leq \|u_n\|$, for $t>s$, we have
\begin{eqnarray*}
\|I^2_n(t)-I^2_n(s)\|^2_{\mathbb{V}'}&=&\|\int^t_sA_1 u_n(r)dr \|^2_{\mathbb{V}'}\\
&\leq& C(t-s)\int^t_s\|A_1 u_n(r)\|^2_{\mathbb{V}'}dr\\
&\leq& C(t-s)\int^t_s\|u_n(r)\|^2dr.
\end{eqnarray*}
Hence, by (\ref{eq-28}), we have for $\alpha\in (0,\frac{1}{2})$,
\begin{eqnarray*}\notag
&&\|I^2_n\|^2_{W^{\alpha,2}([0,T];\mathbb{V}')}\\ \notag
&\leq&\int^T_0 \|I^2_n(t)\|^2_{\mathbb{V}'}dt+\int^T_0\int^T_0\frac{\|I^2_n(t)-I^2_n(s)\|^2_{\mathbb{V}'}}{|t-s|^{1+2\alpha}}dsdt\\
&\leq& C_2(\alpha).
\end{eqnarray*}
Moreover, using (\ref{eq-3}), for $t>s$, we get
\begin{eqnarray*}
\|I^3_n(t)-I^3_n(s)\|^2_{\mathbb{V}'}&=&\|\int^t_sB_n( u_n(r))dr \|^2_{\mathbb{V}'}\\
&\leq& C(t-s)\int^t_s\|B_n( u_n(r))\|^2_{\mathbb{V}'}dr\\
&\leq& C(t-s)\int^t_s|u_n(r)|^4dr\\
&\leq& C(t-s)\sup_{t\in[0,T]}|u_n(t)|^2\int^t_s\|u_n(r)\|^2dr,
\end{eqnarray*}
thus, by (\ref{eq-28}), for $\alpha\in (0,\frac{1}{2})$, we have
\begin{eqnarray*}\notag
\|I^3_n\|^2_{W^{\alpha,2}([0,T];\mathbb{V}')}
\leq C_3(\alpha).
\end{eqnarray*}
Utilizing (\ref{eq-4}), for $t>s$, we deduce that
\begin{eqnarray*}
\|I^4_n(t)-I^4_n(s)\|^2_{\mathbb{V}'}&=&\|\int^t_sM_n( \theta_n(r))dr \|^2_{\mathbb{V}'}\\
&\leq& \left(\int^t_s\|M_n( \theta_n(r))\|_{\mathbb{V}'}dr\right)^2\\
&\leq& \left(\int^t_s\| \theta_n(r)\||\Delta \theta_n(r)|dr\right)^2\\
&\leq& C(t-s)\sup_{t\in[0,T]}\| \theta_n(r)\|^2\int^t_s|\Delta \theta_n(r)|^2dr,
\end{eqnarray*}
hence, by (\ref{eq-28}) and (\ref{eq-49}), for $\alpha\in (0,\frac{1}{2})$, we have
\begin{eqnarray*}\notag
\|I^4_n\|^2_{W^{\alpha,2}([0,T];\mathbb{V}')}
\leq C_4(\alpha).
\end{eqnarray*}
For $I^5_n$, we have
\begin{eqnarray*}
|I^5_n(t)-I^5_n(s)|^2_{\mathbb{H}}&=&\Big|\int^t_sP_n\int_{\mathbb{X}}G(r,u_n(r),v)(g(r,v)-1)\vartheta(dv) dr\Big|^2_{\mathbb{H}}\\
&\leq& \left(\int^t_s\int_{\mathbb{X}}|G(r,u_n(r),v)||g(r,v)-1|\vartheta(dv)dr\right)^2\\
&\leq& \left(\int^t_s\int_{\mathbb{X}}|G(r,v)|_{0,\mathbb{H}}|g(r,v)-1|(1+|u_n(r)|)\vartheta(dv)dr\right)^2\\
&\leq& (1+\sup_{t\in[0,T]}| u_n(r)|^2)\left(\int^t_s\int_{\mathbb{X}}|G(r,v)|_{0,\mathbb{H}}|g(r,v)-1|\vartheta(dv)dr\right)^2\\
&\leq& (1+\sup_{t\in[0,T]}| u_n(r)|^2)\int^T_0\int_{\mathbb{X}}|G(r,v)|_{0,\mathbb{H}}|g(r,v)-1|\vartheta(dv)dr\\
&& \times \int^t_s\int_{\mathbb{X}}|G(r,v)|_{0,\mathbb{H}}|g(r,v)-1|\vartheta(dv)dr,
\end{eqnarray*}
with the help of (\ref{eq-46}) and  (\ref{eq-28}), for $\alpha\in (0,\frac{1}{2})$, we get
\begin{eqnarray*}
\|I^5_n\|^2_{W^{\alpha,2}([0,T];\mathbb{H})}
\leq C_5(\alpha).
\end{eqnarray*}
Based on the above estimates, we complete the proof of (\ref{eq-50}). The proof of (\ref{eq-51}) is similar to (\ref{eq-50}). $\theta_n(t)$ can be written as
\begin{eqnarray*}
\theta_n(t)&=&P_n \theta_0-\int^t_0A_2\theta_n(s)ds-\int^t_0\tilde{B}_n(u_n(s),\theta_n(s))ds-\int^t_0f_n(\theta_n(s))ds\\
&:=&J^1_n+J^2_n(t)+J^3_n(t)+J^4_n(t).
\end{eqnarray*}
It's easy to know $\|J^1_n\|_{H^1}\leq L_1$. For $t>s$, we have
\begin{eqnarray*}
\|J^2_n(t)-J^2_n(s)\|^2_{(H^2)'}&=&\|\int^t_sA_2\theta_n(r)dr \|^2_{(H^2)'}\\
&\leq& C\int^t_s\|A_2\theta_n(r)\|^2_{(H^2)'}dr\\
&\leq& C\int^t_s|\theta_n(r)|^2dr\\
&\leq& C(t-s)\sup_{t\in[0,T]}|\theta_n(t)|^2,
\end{eqnarray*}
thus, by (\ref{eq-27}), for $\alpha\in (0,\frac{1}{2})$, we have
\begin{eqnarray*}\notag
\|J^2_n\|^2_{W^{\alpha,2}([0,T];(H^2)')}
\leq L_2(\alpha).
\end{eqnarray*}
Using Lemma \ref{lem-2}, for $t>s$, we get
\begin{eqnarray*}
\|J^3_n(t)-J^3_n(s)\|^2_{(H^2)'}&=&\|\int^t_s\tilde{B}_n(u_n(r),\theta_n(r))dr \|^2_{(H^2)'}\\
&\leq& C\left(\int^t_s|\tilde{B}_n(u_n(r),\theta_n(r))|dr\right)^2\\
&\leq& C\left(\int^t_s|u_n(r)|^{\frac{1}{2}}\|u_n(r)\|^{\frac{1}{2}}\|\theta_n(r)\|^{\frac{1}{2}}|\Delta\theta_n(r)|^{\frac{1}{2}}dr\right)^2\\
&\leq& C(t-s)\sup_{t\in[0,T]}|u_n(t)|^2\left(\int^t_s\|u_n(r)\|^2dr\right)+C(t-s)\sup_{t\in[0,T]}\|\theta_n(t)\|^2\left(\int^t_s|\Delta \theta_n(r)|^2dr\right),
\end{eqnarray*}
hence, using (\ref{eq-27})-(\ref{eq-28}) and (\ref{eq-49}), for $\alpha\in (0,\frac{1}{2})$, it gives that
\begin{eqnarray*}\notag
\|J^3_n\|^2_{W^{\alpha,2}([0,T];(H^2)')}
\leq L_3(\alpha).
\end{eqnarray*}

Utilizing (\ref{eq-13}), for $t>s$, we deduce that
\begin{eqnarray*}
\|J^4_n(t)-J^4_n(s)\|^2_{(H^2)'}&=&\|\int^t_sf_n(\theta_n(r))dr \|^2_{(H^2)'}\\
&\leq& C(\int^t_s|f_n(\theta_n(r))|dr)^2\\
&\leq& C(t-s)(\int^t_s|f_n(\theta_n(r))|^2dr)\\
&\leq&  C(t-s)[\int^t_s(1+|\theta_n(r)|^{4N+2}_{L^{4N+2}})dr]\\
&\leq&  C(t-s)[\int^t_s(1+\|\theta_n(r)\|^{4N+2}_{H^1})dr]\\
&\leq&  C(t-s)^2+C(t-s)\int^t_s\|\theta_n(r)\|^{4N+2}_{H^1}dr\\
&\leq&  C(t-s)^2+C(t-s)^2\sup_{t\in [0,T]}\|\theta_n(t)\|^{4N+2}_{H^1},
\end{eqnarray*}
where for any $N\in \mathbb{N}^+$, $H^1(\mathbb{O})\hookrightarrow L^{4N+2}(\mathbb{O})$ is used.
By (\ref{eq-49}), for $\alpha\in (0,\frac{1}{2})$, we deduce that
\begin{eqnarray*}\notag
\|J^4_n\|^2_{W^{\alpha,2}([0,T];(H^2)')}
\leq L_4(\alpha).
\end{eqnarray*}
Therefore, collecting all the above estimates, it gives (\ref{eq-51}).

Based on (\ref{eq-50})-(\ref{eq-51}), applying Lemma \ref{lem-5}, we conclude that $u_n$ is compact in $L^2([0,T]; \mathbb{H})\cap C([0,T];\mathbb{V}')$ and $\theta_n$ is compact in $L^2([0,T]; H^1)\cap C([0,T];(H^2)')$.
Moreover, using (\ref{eq-27})-(\ref{eq-28}), we deduce that there exists $(\hat{u},\hat{\theta})$ and a subsequence still denoted by $(u_n,\theta_n)$ such that
\begin{description}
  \item[1.]$\hat{u}\in L^2([0,T]; \mathbb{H})\cap C([0,T];(\mathbb{V})')\cap L^{\infty}([0,T];\mathbb{H})\cap L^2([0,T];\mathbb{V})$,
  \item[2.]$\hat{\theta}\in L^2([0,T]; H^1)\cap C([0,T];(H^2)')\cap L^{\infty}([0,T];H^1)\cap L^2([0,T];H^2)$,
  \item[3.]$u_n\rightarrow \hat{u}$  \ weakly\  star \ in\ $L^{\infty}([0,T];\mathbb{H})$,\quad\quad
$u_n \rightarrow \hat{u}$ \ strongly \ in\ $L^2([0,T]; \mathbb{H})$,

  \item[4.]$u_n \rightarrow \hat{u} $ \ weakly \ in\ $L^2([0,T];\mathbb{V})$,\quad\quad $u_n \rightarrow \hat{u}$ \ strongly \ in\ $C([0,T];(\mathbb{V})')$.
  \item[5.]$\theta_n\rightarrow \hat{\theta}$  \ weakly\  star \ in\ $L^{\infty}([0,T];H^1)$,\quad\quad
$\theta_n \rightarrow \hat{\theta}$ \ strongly \ in\ $L^2([0,T]; H^1)$,
\item[6.]$\theta_n\rightarrow \hat{\theta} $ \ weakly \ in\ $L^2([0,T];H^2)$,\quad\quad $\theta_n \rightarrow \hat{\theta}$ \ strongly \ in\ $C([0,T];(H^2)')$.
\end{description}
Next, we need to show $(\hat{u},\hat{\theta})$ is the unique solution of (\ref{eq-21})-(\ref{eq-22}). We will use the same method as \cite{Y-Z-Z}.

Let $\psi$ be a continuously differential function defined on $[0,T]$ with $\psi(T)=0$. Recall $\{\varrho_j\}_{j\geq 1}$ is an orthonormal eigenfunction of $\mathbb{H}$, which can be viewed as an orthonormal eigenfunction of $\mathbb{V}$. Multiplying (\ref{eq-25})
by $\psi(t)\varrho_j$ and using integration by parts, we obtain
\begin{eqnarray*}
&&-\int^T_0\langle u_n(t), \psi'(t)\varrho_j\rangle dt+\int^T_0\langle u_n(t),\psi(t)A_1\varrho_j\rangle dt\\
&=&\langle P_n u_0, \psi(0)\varrho_j\rangle -\int^T_0\langle P_n B_n(u_n(t)),\psi(t)\varrho_j\rangle dt -\int^T_0\langle P_n M_n(\theta_n(t)),\psi(t)\varrho_j\rangle dt\\
&& \quad
+\int^T_0\langle P_n\int_{\mathbb{X}}G(t, u_n(t),v)(g(t,v)-1)\vartheta(dv), \psi(t)\varrho_j\rangle dt.
\end{eqnarray*}
For every $n>\sup_{m\in \mathbb{N}^+}\sup_{t\in[0,T]}|u_m(t)|^2_{\mathbb{H}}\vee \sup_{m\in \mathbb{N}^+}\sup_{t\in[0,T]}|\theta_m(t)|^2_{L^2}\vee j$, we have
\begin{eqnarray*}
&&-\int^T_0\langle u_n(t), \psi'(t)\varrho_j\rangle dt+\int^T_0\langle u_n(t),\psi(t) A_1\varrho_j\rangle dt\\
&=&\langle u_0, \psi(0)\varrho_j\rangle -\int^T_0\langle B(u_n(t)),\psi(t)\varrho_j\rangle dt -\int^T_0\langle M(\theta_n(t)),\psi(t)\varrho_j\rangle dt\\
&& \quad
+\int^T_0\langle \int_{\mathbb{X}}G(t, u_n(t),v)(g(t,v)-1)\vartheta(dv), \psi(t)\varrho_j\rangle dt.
\end{eqnarray*}
Denote the above equality by symbols
 \[
J_1(T)+J_2(T)=J_3+J_4(T)+J_5(T)+J_6(T).
\]
Since $u_n \rightarrow \hat{u}$ \ strongly \ in\ $C([0,T];\mathbb{V}')$, we have
\[
J_1(T)\rightarrow -\int^T_0\langle \hat{u}(t), \psi'(t)\varrho_j\rangle dt.
\]
With the aid of Cauchy-Schwarz inequality and $u_n \rightarrow \hat{u}$  strongly  in $L^2([0,T]; \mathbb{H})$, we get
\[
J_2(T)\rightarrow \int^T_0\langle\hat{u}(t),\psi(t) A_1 \varrho_j\rangle dt.
\]
 By the triangle inequality and (\ref{eq-3}), we have
\begin{eqnarray*}
&&\Big|\int^T_0\langle B(u_n(t)),\psi(t)\varrho_j\rangle dt-\int^T_0\langle B(\hat{u}(t)),\psi(t)\varrho_j\rangle dt\Big|\\
&\leq & \int^T_0|\langle B(u_n(t)-\hat{u}(t), u_n(t)),\psi(t)\varrho_j\rangle|dt+\int^T_0|\langle B(\hat{u}(t), u_n(t)-\hat{u}(t)),\psi(t)\varrho_j\rangle|dt\\
&\leq & 2\int^T_0| u_n(t)-\hat{u}(t)|(| u_n(t)|+|u(t)|)\|\psi(t)\varrho_j\|_{\mathbb{V}}dt\\
&\leq & 2(\int^T_0| u_n(t)-\hat{u}(t)|^2dt)^{\frac{1}{2}}\sup_{t\in [0,T]}(| u_n(t)|+|u(t)|)|\psi(t)|,
\end{eqnarray*}
hence,
\[
J_4(T)\rightarrow \int^T_0\langle B(\hat{u}(t)),\psi(t)\varrho_j\rangle  dt.
\]
Using (\ref{eq-4}), we get
\begin{eqnarray*}
&&|\int^T_0\langle M(\theta_n(t)),\psi(t)\varrho_j\rangle dt-\int^T_0\langle M(\hat{\theta}(t)),\psi(t)\varrho_j\rangle dt|\\
&\leq & \int^T_0|\langle M(\theta_n(t)-\hat{\theta}(t), \theta_n(t)),\psi(t)\varrho_j\rangle|dt+\int^T_0|\langle M(\hat{\theta}(t), \theta_n(t)-\hat{\theta}(t)),\psi(t)\varrho_j\rangle|dt\\
&\leq & \int^T_0(\|\theta_n(t)-\hat{\theta}(t)\|^{\frac{1}{2}}|\Delta(\theta_n(t)-\hat{\theta}(t))|^{\frac{1}{2}}\|\theta_n(t)\|^{\frac{1}{2}}|\Delta \theta_n(t)|^{\frac{1}{2}})\|\psi(t)\varrho_j\|_{\mathbb{V}}dt\\
&&\quad +\int^T_0(\|\theta_n(t)-\hat{\theta}(t)\|^{\frac{1}{2}}|\Delta(\theta_n(t)-\hat{\theta}(t))|^{\frac{1}{2}}\|\hat{\theta}(t)\|^{\frac{1}{2}}| \Delta \hat{\theta}(t)|^{\frac{1}{2}})\|\psi(t)\varrho_j\|_{\mathbb{V}}dt\\
&\leq & \sup_{t\in[0,T]}|\psi(t)|(\int^T_0 \|\theta_n-\hat{\theta}\|^2ds)^{\frac{1}{4}}[(\int^T_0 |\Delta \theta_n|^2ds)^{\frac{1}{4}}+(\int^T_0 |\Delta \hat{\theta}|^2ds)^{\frac{1}{4}}][T^{\frac{1}{2}}\sup_{t\in[0,T]}\|\theta_n\|^{\frac{1}{2}}(\int^T_0 |\Delta \theta_n|^2ds)^{\frac{1}{4}}]\\
&&\ +\sup_{t\in[0,T]}|\psi(t)|(\int^T_0 \|\theta_n-\hat{\theta}\|^2ds)^{\frac{1}{4}}[(\int^T_0 |\Delta \theta_n|^2ds)^{\frac{1}{4}}+(\int^T_0 |\Delta \hat{\theta}|^2ds)^{\frac{1}{4}}][T^{\frac{1}{2}}\sup_{t\in[0,T]}\|\hat{\theta}\|^{\frac{1}{2}}(\int^T_0 |\Delta \hat{\theta}|^2ds)^{\frac{1}{4}}].
\end{eqnarray*}
With the help of (\ref{eq-28}) and $\theta_n \rightarrow \hat{\theta}$ \ strongly \ in\ $L^2([0,T]; H^1)$, we conclude that
\[
J_5(T)\rightarrow \int^T_0\langle M(\hat{\theta}(t)),\psi(t)\varrho_j\rangle dt.
\]

The proof of $J_6(T)\rightarrow\int^T_0\langle \int_{\mathbb{X}}G(t, \hat{u}(t),v)(g(t,v)-1)\vartheta(dv), \psi(t)\varrho_j\rangle dt$ is the same as (4.25) in \cite{Y-Z-Z}, we omit it.
Based on the above steps, we conclude that for any $j\geq 1$,
\begin{eqnarray}\notag
&&-\int^T_0\langle \hat{u}(t), \psi'(t)\varrho_j\rangle dt+\int^T_0\langle \hat{u}(t),\psi(t) A_1\varrho_j\rangle dt\\ \notag
&=&\langle u_0, \psi(0)\varrho_j\rangle -\int^T_0\langle B(\hat{u}(t)),\psi(t)\varrho_j\rangle dt -\int^T_0\langle M(\hat{\theta}(t)),\psi(t)\varrho_j\rangle dt\\
\label{eq-53}
&& \quad
+\int^T_0\langle \int_{\mathbb{X}}G(t, \hat{u}(t),v)(g(t,v)-1)\vartheta(dv), \psi(t)\varrho_j\rangle dt.
\end{eqnarray}
Actually, (\ref{eq-53}) holds for any $\zeta$, which is a finite linear combination of $\varrho_j$. That is

\begin{eqnarray}\notag
&&-\int^T_0\langle \hat{u}(t), \psi'(t)\zeta\rangle dt+\int^T_0\langle \hat{u}(t),\psi(t) A_1\zeta\rangle dt\\ \notag
&=&\langle u_0, \psi(0)\zeta\rangle -\int^T_0\langle B(\hat{u}(t)),\psi(t)\zeta\rangle dt -\int^T_0\langle M(\hat{\theta}(t)),\psi(t)\zeta\rangle dt\\
\label{eq-54}
&& \quad
+\int^T_0\langle \int_{\mathbb{X}}G(t, \hat{u}(t),v)(g(t,v)-1)\vartheta(dv), \psi(t)\zeta\rangle dt.
\end{eqnarray}
Since $\mathbb{V}$ is dense in $\mathbb{H}$, we get
\begin{eqnarray}\label{eq-55}
d\hat{u}(t)+A_1 \hat{u}(t)dt+B(\hat{u}(t))dt+ M(\hat{\theta}(t))dt=\int_{\mathbb{X}}G(t, \hat{u}(t),v)(g(t,v)-1)\vartheta(dv)dt
\end{eqnarray}
holds as an equality in distribution in $L^2([0,T];\mathbb{H}')$.

 Finally, it remains to prove $\hat{u}(0)={u}_0$. Multiplying (\ref{eq-55}) with the same $\psi(t)$ as above and integrating with respect to $t$. By integration by parts, we have
\begin{eqnarray}\notag
&&-\int^T_0\langle \hat{u}(t), \psi'(t)\zeta\rangle dt+\int^T_0\langle \hat{u}(t),\psi(t) A_1\zeta\rangle dt\\ \notag
&=&\langle \hat{u}_0, \psi(0)\zeta\rangle -\int^T_0\langle B(\hat{u}(t)),\psi(t)\zeta\rangle dt -\int^T_0\langle M(\hat{\theta}(t)),\psi(t)\zeta\rangle dt\\
\label{eq-56}
&& \quad
+\int^T_0\langle \int_{\mathbb{X}}G(t, \hat{u}(t),v)(g(t,v)-1)\vartheta(dv), \psi(t)\zeta\rangle dt.
\end{eqnarray}
 By comparison with (\ref{eq-54}), it gives $\langle u_0- \hat{u}_0, \psi(0)\zeta\rangle=0, \forall \zeta\in \mathbb{V}$. Choosing $\psi$ such that $\psi(0)\neq 0$, then
\[
(\hat{u}(0)-{u}_0,\zeta)=0,\  \forall \zeta\in  \mathbb{V}.
\]
Since $ \mathbb{V}$ is dense in $ \mathbb{H}$, we have $\hat{u}(0)=u_0$.

Using the same method as above, we can obtain the following equality holds
\begin{eqnarray*}
&&-\int^T_0\langle \hat{\theta}(t), \psi'(t)\varsigma_j\rangle dt+\int^T_0\langle \hat{\theta}(t),\psi(t)A_2 \varsigma_j\rangle dt\\
&=&\langle  \hat{\theta}_0, \psi(0)\varsigma_j\rangle -\int^T_0\langle \tilde{B}(\hat{u}(t),\hat{\theta}(t)),\psi(t)\varsigma_j\rangle dt -\int^T_0\langle f_n(\hat{\theta}(t)),\psi(t)\varsigma_j\rangle dt.
\end{eqnarray*}
Therefore, $(\hat{u},\hat{\theta})$ satisfies (\ref{eq-21})-(\ref{eq-22}).

(Continuity)\  According to Lemma \ref{lem-4}, we need to show
\[
\frac{d\hat{u}}{dt}\in L^2([0,T];\mathbb{V}'),\quad \frac{d\hat{\theta}}{dt}\in L^2([0,T];(H^2)').
\]
The proof is similar to the proof process of (\ref{eq-50})-(\ref{eq-51}), we omit it. Thus, we obtain
\[
\hat{u}\in C([0,T];\mathbb{H}),\quad \hat{\theta}\in  C([0,T];H^1).
\]

(Uniqueness) \ Assume $(u_1,\theta_1), (u_2,\theta_2)$ are two solutions of  (\ref{eq-21})-(\ref{eq-22}). Let $u=u_1-u_2, \theta=\theta_1-\theta_2$, then $(u_0,\theta_0)=(0,0)$. From (\ref{eq-21})-(\ref{eq-22}), we have
\begin{eqnarray*}
&&du(t)+A_1 u(t)dt+(B(u(t),u_1(t))+B(u_2(t), u(t)))dt
= -(M(\theta(t),\theta_1(t))+M(\theta_2(t),\theta(t)))dt\\ \notag
 &&\quad +\int_{\mathbb{X}}(G(t,u_1(t),v)-G(t,u_2(t),v))(g(t,v)-1)\vartheta(dv)dt,\\
&&d\theta(t)+A_2 \theta(t)dt+(\tilde{B}(u(t),\theta_1(t))+\tilde{B}(u_2(t), \theta(t)))dt
= -(f(\theta_1(t))-f(\theta_2(t)))dt.
\end{eqnarray*}
Define
\[
\Upsilon(t):=\exp\left\{-2\int^t_0(\xi_1(s)+\xi_2(s)+\xi_3(s))ds\right\}, \quad \forall t>0,
\]
where
\begin{eqnarray*}
\xi_1(s)&=&c(\kappa_3)|u_1|^2\|u_1\|^2+C(\kappa_9)\|\theta_1\|^2+C(\kappa_{10},\kappa_{11})\|\theta_1\|^2|\Delta \theta_1|^2,\\
\xi_2(s)&=&c(\kappa_1)+c(\kappa_2)\beta(\theta_1,\theta_2),\\
\xi_3(s)&=&c(\kappa_4,\kappa_5)\|\theta_2\|^2|\Delta \theta_2|^2+c(\kappa_6,\kappa_8)\|\theta_1\|^2|\Delta \theta_1|^2+c(\kappa_7)|u_2|^2\|u_2\|^2+c(\kappa_2)\beta(\theta_1,\theta_2),
\end{eqnarray*}
with
\[
\beta(\theta_1,\theta_2)=C(1+|\theta_1|^{2N}_{L^{4N+2}}+|\theta_2|^{2N}_{L^{4N+2}})^2.
\]

Using Lemma \ref{lem-2}, we get
\begin{eqnarray}\label{eq-57}
d[\Upsilon(t)|\theta(t)|^2]=-2\Upsilon(t)[\|\theta(t)\|^2+\langle \tilde{B}(u(t),\theta_1(t))+f(\theta_1(t))-f(\theta_2(t)), \theta(t)\rangle]dt+\Upsilon'(t)|\theta(t)|^2dt,
\end{eqnarray}
and
\begin{eqnarray}\notag
d[\Upsilon(t)\|\theta(t)\|^2]&=&2\Upsilon(t)[-|\Delta \theta(t)|^2+\langle \tilde{B}(u(t),\theta_1(t))+\tilde{B}(u_2(t),\theta(t))+f(\theta_1(t))-f(\theta_2(t)), \Delta \theta(t)\rangle]dt\\
\label{eq-58}
&&\ +\Upsilon'(t)\|\theta(t)\|^2dt.
\end{eqnarray}
By Lemma \ref{lem-1}, we obtain
\begin{eqnarray}\notag
&&d[\Upsilon(t)|u(t)|^2]\\ \notag
&=&-2\Upsilon(t)[\|u(t)\|^2+\langle {B}(u(t),u_1(t))+M(\theta(t),\theta_1(t))+ M(\theta_2(t),\theta(t)), u(t)\rangle]dt\\
\label{eq-59}
&&\ +2\Upsilon(t)\langle\int_{\mathbb{X}}(G(t,u_1(t),v)-G(t,u_2(t),v))(g(t,v)-1)\vartheta(dv), u(t)\rangle dt+\Upsilon'(t)|u(t)|^2dt.
\end{eqnarray}
Referring to (9.13) in \cite{B-M-P}, it gives
\begin{eqnarray*}
|\langle B(u,u_1),u\rangle|&\leq & k_3\|u\|^2+C(\kappa_3)|u_1|^2\|u_1\|^2|u|^2,\\
|\langle M(\theta_2, \theta), u\rangle|&\leq & \kappa_4\|u\|^2+\kappa_5|\Delta \theta|^2+C(\kappa_4,\kappa_5)\|\theta_2\|^2|\Delta \theta_2|^2\|\theta\|^2,\\
|\langle M(\theta, \theta_1), u\rangle|&\leq & \kappa_8\|u\|^2+\kappa_6|\Delta \theta|^2+C(\kappa_6,\kappa_8)\|\theta_1\|^2|\Delta \theta_1|^2\|\theta\|^2,\\
|\langle \tilde{B}(u_2,\theta),\Delta \theta\rangle|&\leq & \kappa_7 |\Delta \theta|^2+C(\kappa_7)|u_2|^2\|u_2\|^2\|\theta\|^2,\\
|\langle \tilde{B}(u, \theta_1),  \theta\rangle|&\leq & \kappa_9|\Delta \theta|^2+C(\kappa_9)|u|^2\|\theta_1\|^2,\\
|\langle \tilde{B}(u,\theta_1), \Delta \theta\rangle|&\leq & \kappa_{10}|\Delta \theta|^2+\kappa_{11}\|u\|^2+C(\kappa_{10},\kappa_{11})|u|^2\|\theta_1\|^2|\Delta \theta_1|^2.
\end{eqnarray*}
Adding up (\ref{eq-57}),(\ref{eq-58}) and (\ref{eq-59}), then using the above estimates and (\ref{eq-10})-(\ref{eq-11}), we get
\begin{eqnarray*}
&&d[\Upsilon(t)(|u(t)|^2+|\theta(t)|^2+\|\theta(t)\|^2)]+2\Upsilon(t)[\|u(t)\|^2+\|\theta(t)\|^2+|\Delta\theta(t)|^2]dt\\
&\leq& 2\Upsilon(t)[\xi_1(t)|u(t)|^2+\xi_2(t)|\theta(t)|^2+\xi_3(t)\|\theta(t)\|^2]dt\\
&&\ +2\Upsilon(t)[L_1\|u(t)\|^2+L_2\|\theta(t)\|^2+L_3|\Delta\theta(t)|^2]dt\\
&&\ +2\Upsilon(t)|u(t)|^2\int_{\mathbb{X}}|G(t,v)|_{1,\mathbb{H}}|g(t,v)-1|\vartheta(dv)dt\\
&&\ +\Upsilon'(t)[|u(t)|^2+|\theta(t)|^2+\|\theta(t)\|^2]dt,
\end{eqnarray*}
where
\[
L_1=\kappa_3+\kappa_4+\kappa_8+\kappa_{11},\quad L_2=\kappa_1,\quad L_3=\kappa_2+\kappa_5+\kappa_6+\kappa_{7}+\kappa_9+\kappa_{10}.
\]
Choosing $\kappa_3=\kappa_4=\kappa_8=\kappa_{11}=\frac{1}{8}$,  $\kappa_1=\frac{1}{2}$ and $\kappa_2=\kappa_5=\kappa_6=\kappa_{7}=\kappa_9=\kappa_{10}=\frac{1}{12}$, we have
\begin{eqnarray*}
&&d[\Upsilon(t)(|u(t)|^2+|\theta(t)|^2+\|\theta(t)\|^2)]+\Upsilon(t)[\|u(t)\|^2+\|\theta(t)\|^2+|\Delta\theta(t)|^2]dt\\
&\leq& 2\Upsilon(t)[\xi_1(t)|u(t)|^2+\xi_2(t)|\theta(t)|^2+\xi_3(t)\|\theta(t)\|^2]dt\\
&&\ +2\Upsilon(t)|u(t)|^2\int_{\mathbb{X}}|G(t,v)|_{1,\mathbb{H}}|g(t,v)-1|\vartheta(dv)dt\\
&&\ +\Upsilon'(t)[|u(t)|^2+|\theta(t)|^2+\|\theta(t)\|^2]dt.
\end{eqnarray*}
By the choice of $\Upsilon(t)$, we deduce that
\[
2\Upsilon(t)[\xi_1(t)|u(t)|^2+\xi_2(t)|\theta(t)|^2+\xi_3(t)\|\theta(t)\|^2]+\Upsilon'(t)[|u(t)|^2+|\theta(t)|^2+\|\theta(t)\|^2]\leq0.
\]
Hence, we conclude that
\begin{eqnarray*}
&&d[\Upsilon(t)(|u(t)|^2+|\theta(t)|^2+\|\theta(t)\|^2)]+\Upsilon(t)[\|u(t)\|^2+\|\theta(t)\|^2+|\Delta\theta(t)|^2]dt\\
&\leq& 2[\Upsilon(t)(|u(t)|^2+|\theta(t)|^2+\|\theta(t)\|^2)]\int_{\mathbb{X}}|G(t,v)|_{1,\mathbb{H}}|g(t,v)-1|\vartheta(dv)dt.
\end{eqnarray*}
Applying Gronwall inequality to the above inequality and using (\ref{eq-46}), we obtain the uniqueness. Up to now, we complete the proof of Theorem \ref{thm-2}.

\end{proof}

\section{Large deviations}
This section is devoted to the proof of the main result.
According to Theorem \ref{thm-1}, we need to prove (i) and (ii) in \textbf{Condition A}.

Firstly, we prove (i) in \textbf{Condition A}. For $g\in S$, from  Theorem \ref{thm-2}, we can define
\begin{eqnarray*}
\mathcal{G}^0(\vartheta^g_T)=(u^g,\theta^g).
\end{eqnarray*}
\begin{prp}\label{prp-1}
For any $M\in \mathbb{N}^+$, and $\{g_n\}_{n\geq 1}\subset S^M, g\in S^M$ satisfying $g_n\rightarrow g$ as $n\rightarrow \infty$. Then
\begin{eqnarray*}
\mathcal{G}^0(\vartheta^{g_n}_T)\rightarrow \mathcal{G}^0(\vartheta^{g}_T) \ in\ C([0,T];\mathbb{H})\times C([0,T];H^1).
\end{eqnarray*}
\end{prp}
\begin{proof}
Recall that $\mathcal{G}^0(\vartheta^{g_n}_T)=(u^{g_n},\theta^{g_n})$. For simplicity, denote $(u_n,\theta_n)=(u^{g_n},\theta^{g_n})$.

Using similar method as Theorem \ref{thm-2} and by Lemma \ref{lem-6}, we can prove that
\begin{eqnarray}\label{eq-27-1}
\sup_{n} \left(\sup_{s\in [0,T]}|\theta_n(s)|^p_{L^2}+\int^T_0|\theta_n(s)|^{p-2}_{L^2}(\|\theta_n(s)\|^2+|\theta_n(s)|^{2N+2}_{L^{2N+2}})ds\right)\leq C(p,M),
\end{eqnarray}
for any $p\geq 1$,
\begin{eqnarray}\label{eq-28-1}
\sup_{n}\left( \sup_{s\in [0,T]}(\Psi(\theta_n(s))+|u_n(s)|^2)^p+\Big(\int^T_0(\|u_n(s)\|^2+|\Delta \theta_n(s)-f(\theta_n(s))|^2)ds\Big)^p\right)\leq C(p,M),
\end{eqnarray}
where $\Psi(\theta_n(s)):=\frac{1}{2}\|\theta_n(s))\|^2+\frac{1}{2}\int_{\mathbb{O}}F(|\theta_n(s)|^2)dx$ and $C(p,M)$ is independent of $n$.
Moreover, for $\alpha\in (0,\frac{1}{2})$, there exist $C(\alpha), L(\alpha)$ such that
\begin{eqnarray}\notag
\sup_{n\geq 1}\|u_n\|^2_{W^{\alpha,2}([0,T];\mathbb{V}')}\leq C(\alpha),\quad
\sup_{n\geq 1}\|\theta_n\|^2_{W^{\alpha,2}([0,T];(H^2)')}\leq L(\alpha).
\end{eqnarray}
Hence, we deduce from Lemma \ref{lem-5} that there exists an element $(u,\theta)$ and a subsequence still denoted by $(u_n,\theta_n)$ such that
\begin{description}
  \item[1.]${u}\in L^2([0,T]; \mathbb{H})\cap C([0,T];(\mathbb{V})')\cap L^{\infty}([0,T];\mathbb{H})\cap L^2([0,T];\mathbb{V})$,
  \item[2.]${\theta}\in L^2([0,T]; H^1)\cap C([0,T];(H^2)')\cap L^{\infty}([0,T];H^1)\cap L^2([0,T];H^2)$,
  \item[3.]$u_n\rightarrow {u}$  \ weakly\  star \ in\ $L^{\infty}([0,T];\mathbb{H})$,\quad\quad
$u_n \rightarrow {u}$ \ strongly \ in\ $L^2([0,T]; \mathbb{H})$,

  \item[4.]$u_n \rightarrow {u} $ \ weakly \ in\ $L^2([0,T];\mathbb{V})$,\quad\quad $u_n \rightarrow {u}$ \ strongly \ in\ $C([0,T];(\mathbb{V})')$.
  \item[5.]$\theta_n\rightarrow {\theta}$  \ weakly\  star \ in\ $L^{\infty}([0,T];H^1)$,\quad\quad
$\theta_n \rightarrow {\theta}$ \ strongly \ in\ $L^2([0,T]; H^1)$,
\item[6.]$\theta_n\rightarrow {\theta} $ \ weakly \ in\ $L^2([0,T];H^2)$,\quad\quad $\theta_n \rightarrow {\theta}$ \ strongly \ in\ $C([0,T];(H^2)')$.
\end{description}
We will prove $(u,\theta)=(u^g, \theta^g)$.

Let $\psi$ be a continuously differential function defined on $[0,T]$ with $\psi(T)=0$.  Multiplying $u_n(t)$
by $\psi(t)\varrho_j$ and using integration by parts, for every $n>(\sup_{m\in \mathbb{N}^+}\sup_{t\in[0,T]}|u_m(t)|^2_{\mathbb{H}})\vee (\sup_{m\in \mathbb{N}^+}\sup_{t\in[0,T]}|\theta_m(t)|^2_{L^2})\vee j$, we obtain
\begin{eqnarray*}
&&-\int^T_0\langle u_n(t), \psi'(t)\varrho_j\rangle dt+\int^T_0\langle u_n(t),\psi(t)A_1\varrho_j\rangle dt\\
&=&\langle  u_0, \psi(0)\varrho_j\rangle -\int^T_0\langle B(u_n(t)),\psi(t)\varrho_j\rangle dt -\int^T_0\langle M(\theta_n(t)),\psi(t)\varrho_j\rangle dt\\
&& \quad
+\int^T_0\langle \int_{\mathbb{X}}G(t, u_n(t),v)(g_n(t,v)-1)\vartheta(dv), \psi(t)\varrho_j\rangle dt.
\end{eqnarray*}
Utilizing the same method as Theorem \ref{thm-2}, we deduce that
\begin{eqnarray*}
&&-\int^T_0\langle u_n(t), \psi'(t)\varrho_j\rangle dt+\int^T_0\langle u_n(t),\psi(t)A_1\varrho_j\rangle dt\\
&&\ \ -\langle  u_0, \psi(0)\varrho_j\rangle +\int^T_0\langle B(u_n(t)),\psi(t)\varrho_j\rangle dt +\int^T_0\langle M(\theta_n(t)),\psi(t)\varrho_j\rangle dt\\
& \rightarrow& -\int^T_0\langle u(t), \psi'(t)\varrho_j\rangle dt+\int^T_0\langle u(t),\psi(t)A_1\varrho_j\rangle dt\\
&&\ \ -\langle  u_0, \psi(0)\varrho_j\rangle +\int^T_0\langle B(u(t)),\psi(t)\varrho_j\rangle dt +\int^T_0\langle M(\theta(t)),\psi(t)\varrho_j\rangle dt.
\end{eqnarray*}
For the remain term $\int^T_0\langle \int_{\mathbb{X}}G(t, u_n(t),v)(g_n(t,v)-1)\vartheta(dv), \psi(t)\varrho_j\rangle dt$, referring to Proposition 4.1. in \cite{Z-Z}, it gives that
\begin{eqnarray*}
&&\int^T_0\langle \int_{\mathbb{X}}G(t, u_n(t),v)(g_n(t,v)-1)\vartheta(dv), \psi(t)\varrho_j\rangle dt\\
& \rightarrow &  \int^T_0\langle \int_{\mathbb{X}}G(t, u(t),v)(g(t,v)-1)\vartheta(dv), \psi(t)\varrho_j\rangle dt.
\end{eqnarray*}
Therefore, we conclude that
\begin{eqnarray*}
&&-\int^T_0\langle u(t), \psi'(t)\varrho_j\rangle dt+\int^T_0\langle u(t),\psi(t)A_1\varrho_j\rangle dt\\
&=&\langle  u_0, \psi(0)\varrho_j\rangle -\int^T_0\langle B(u(t)),\psi(t)\varrho_j\rangle dt -\int^T_0\langle M(\theta(t)),\psi(t)\varrho_j\rangle dt\\
&& \quad
+\int^T_0\langle \int_{\mathbb{X}}G(t, u(t),v)(g(t,v)-1)\vartheta(dv), \psi(t)\varrho_j\rangle dt.
\end{eqnarray*}
Similarly, we can obtain
\begin{eqnarray*}
&&-\int^T_0\langle \theta(t), \psi'(t)\varsigma_j\rangle dt+\int^T_0\langle \theta(t),\psi(t)A_2\varsigma_j\rangle dt\\
&=&\langle  \theta_0, \psi(0)\varsigma_j\rangle -\int^T_0\langle \tilde{B}(u(t),\theta(t)),\psi(t)\varsigma_j\rangle dt -\int^T_0\langle f(\theta(t)),\psi(t)\varsigma_j\rangle dt.
\end{eqnarray*}
Thus, we get $(u,\theta)=(u^g,\theta^g)$.

Next, we prove $(u_n, \theta_n)\rightarrow (u,\theta)$ in $C([0,T];\mathbb{H})\times C([0,T];H^1)$.
Let $w_n=u_n-u, r_n=\theta_n-\theta$. Then
\begin{eqnarray*}
&&dw_n+A_1 w_ndt+[B(w_n,u_n)+B(u,w_n)]dt\\
&=&-[M(r_n,\theta_n)+M(\theta,r_n)]dt +\int_{\mathbb{X}}\Big(G(t,u_n(t),v)(g_n(t,v)-1)-G(t,u(t),v)(g(t,v)-1)\Big)\vartheta(dv)dt,
\end{eqnarray*}
and
\begin{eqnarray*}
dr_n+A_2 r_ndt+[\tilde{B}(w_n,\theta_n)+\tilde{B}(u,r_n)]dt+(f(\theta_n)-f(\theta))dt=0.
\end{eqnarray*}
Using the same method as  the 'Uniqueness' part in Theorem \ref{thm-2} and referring to the estimation of $I^{m'}_5$ of (4.54) in \cite{Z-Z}, we obtain
\begin{eqnarray*}
\sup_{t\in[0,T]}(|w_n(t)|^2+|r_n(t)|^2+\|r_n(t)\|^2)=0,
\end{eqnarray*}
which implies the desired result.

\end{proof}

Recall  $\mathcal{G}^{\varepsilon}(\varepsilon {N}^{\varepsilon^{-1}})=(u^{\varepsilon}(\cdot), \theta^{\varepsilon}(\cdot))$.
Let $\varphi_{\varepsilon}\in \mathcal{U}^M$ and $\vartheta_{\varepsilon}=\frac{1}{\varphi_{\varepsilon}}$. The following lemma was proved by Budhiraja et al. \cite{B-D-M}.
\begin{lemma}\label{lem-9}
\begin{eqnarray*}
\mathcal{E}^{\varepsilon}_t(\vartheta_{\varepsilon})&:=&\exp\Big\{\int_{(0,t)\times \mathbb{X}\times [0,\varepsilon^{-1}]}\log(\vartheta_{\varepsilon}(s,x)) \bar{N}(dsdxdr)\\
&&\ +\int_{(0,t)\times \mathbb{X}\times [0,\varepsilon^{-1}]}(-\vartheta_{\varepsilon}(s,x)+1) \bar{\vartheta}_T(dsdxdr)\Big\},
\end{eqnarray*}
is an $\{\bar{\mathcal{F}}_t\}-$martingale. Then
\[
\mathbb{Q}^{\varepsilon}_t(G)=\int_{G}\mathcal{E}^{\varepsilon}_t(\vartheta_{\varepsilon})d \bar{\mathbb{P}},\quad {\rm{for }}\ G\in \mathcal{B}(\bar{\mathbb{M}})
\]
defines a probability measure on $\bar{\mathbb{M}}$.
\end{lemma}
Since $\varepsilon N^{\varepsilon^{-1}\varphi_{\varepsilon}}$ under $\mathbb{Q}^{\varepsilon}_T$ has the same law as that of $\varepsilon N^{\varepsilon^{-1}}$ under $\bar{\mathbb{P}}$, it follows that
there exists a unique solution to the following controlled stochastic evolution equations $(\tilde{u}^{\varepsilon},\tilde{\theta}^{\varepsilon})$:
\begin{eqnarray}\notag
\tilde{u}^{\varepsilon}(t)&=&u_0-\int^t_0 A_1 \tilde{u}^{\varepsilon}(s)ds-\int^t_0 B(\tilde{u}^{\varepsilon}(s))ds\\ \notag
&& \ -\int^t_0 M(\tilde{\theta}^{\varepsilon}(s))ds+\int^t_0\int_{\mathbb{X}}G(s,\tilde{u}^{\varepsilon}(s),v)(\varepsilon N^{\varepsilon^{-1}\varphi_{\varepsilon}}(dsdv)-\vartheta(dv)ds)\\ \notag
&=& u_0-\int^t_0 A_1 \tilde{u}^{\varepsilon}(s)ds-\int^t_0 B(\tilde{u}^{\varepsilon}(s))ds\\ \notag
&& \ -\int^t_0 M(\tilde{\theta}^{\varepsilon}(s))ds+\int^t_0\int_{\mathbb{X}}G(s,\tilde{u}^{\varepsilon}(s),v)(\varphi_{\varepsilon}(s,v)-1)\vartheta(dv)ds)\\
\label{eq-67}
&&\ +\int^t_0\int_{\mathbb{X}}\varepsilon G(s,\tilde{u}^{\varepsilon}(s),v)( N^{\varepsilon^{-1}\varphi_{\varepsilon}}(dsdv)-\varepsilon^{-1}\varphi_{\varepsilon}(s,v)\vartheta(dv)ds),\\
\label{eq-66}
\tilde{\theta}^{\varepsilon}(t)&=& \theta_0-\int^t_0 A_2 \tilde{\theta}^{\varepsilon}(s)ds-\int^t_0 \tilde{B}(\tilde{u}^{\varepsilon}(s),\tilde{\theta}^{\varepsilon}(s))ds -\int^t_0 f(\tilde{\theta}^{\varepsilon}(s))ds,
\end{eqnarray}
and we have
\begin{eqnarray}\label{eq-63}
\mathcal{G}^{\varepsilon}(\varepsilon N^{\varepsilon^{-1}\varphi_{\varepsilon}})=(\tilde{u}^{\varepsilon},\tilde{\theta}^{\varepsilon}).
\end{eqnarray}
Before proving (ii) in \textbf{Condition A}, we make a priori estimates of $(\tilde{u}^{\varepsilon},\tilde{\theta}^{\varepsilon})$.
\begin{lemma}\label{lem-10}
There exists $\varepsilon_0>0$ such that for any $p\geq 1$,
\begin{eqnarray}\label{eq-61}
\sup_{0<\varepsilon<\varepsilon_0}\Big[\mathbb{E}\sup_{t\in[0,T]}|\tilde{\theta}^{\varepsilon}(t)|^p_{\mathbb{H}}+\mathbb{E}\int^T_0|\tilde{\theta}^{\varepsilon}(s)|^{p-2}_{L^2}(\|\tilde{\theta}^{\varepsilon}(s)\|^2
+|\tilde{\theta}^{\varepsilon}(s)|^{2N+2}_{L^{2N+2}})ds\Big]\leq C(p,M),
\end{eqnarray}
and
\begin{eqnarray}\label{eq-62}
\sup_{0<\varepsilon<\varepsilon_0}\left( \mathbb{E}\sup_{s\in [0,T]}(\Psi(\tilde{\theta}^{\varepsilon}(s))+|\tilde{u}^{\varepsilon}(s)|^2)^p+\mathbb{E}\Big(\int^T_0(\|\tilde{u}^{\varepsilon}(s)\|^2+|\Delta \tilde{\theta}^{\varepsilon}(s)-f(\tilde{\theta}^{\varepsilon}(s))|^2)ds\Big)^p\right)\leq C(p,M),
\end{eqnarray}
where $\Psi(\tilde{\theta}^{\varepsilon}(s)):=\frac{1}{2}\|\tilde{\theta}^{\varepsilon}(s))\|^2+\frac{1}{2}\int_{\mathbb{O}}F(|\tilde{\theta}^{\varepsilon}(s)|^2)dx.$ Moreover, for $\alpha\in (0,\frac{1}{2})$, there exists constants $L(\alpha), R(\alpha)>0$ such that
\begin{eqnarray}\label{eq-60}
\sup_{0<\varepsilon<\varepsilon_0}\mathbb{E}\|\tilde{u}^{\varepsilon}\|_{W^{\alpha,2}([0,T];\mathbb{V}')}\leq L(\alpha),\quad
\sup_{0<\varepsilon<\varepsilon_0}\mathbb{E}\|\tilde{\theta}^{\varepsilon}\|_{W^{\alpha,2}([0,T];(H^2)')}\leq R(\alpha).
\end{eqnarray}
\end{lemma}

\begin{proof}
Combing Theorem \ref{thm-2} in our paper and Proposition 5.4-5.5 proved by Brze\'{z}niak et al. in \cite{B-M-P}, we can obtain the estimates of (\ref{eq-61}) and (\ref{eq-62}).

Let $\varphi_{\varepsilon}\in \mathcal{U}^M$, we have
\begin{eqnarray*}
\tilde{u}^{\varepsilon}(t)&=& u_0-\int^t_0 A_1 \tilde{u}^{\varepsilon}(s)ds-\int^t_0 B(\tilde{u}^{\varepsilon}(s))ds\\
&& \ -\int^t_0 M(\tilde{\theta}^{\varepsilon}(s))ds+\int^t_0\int_{\mathbb{X}}G(s,\tilde{u}^{\varepsilon}(s),v)(\varphi_{\varepsilon}(s,v)-1)\vartheta(dv)ds)\\
&&\ +\int^t_0\int_{\mathbb{X}}\varepsilon G(s,\tilde{u}^{\varepsilon}(s),v)( N^{\varepsilon^{-1}\varphi_{\varepsilon}}(dsdv)-\varepsilon^{-1}\varphi_{\varepsilon}(s,v)\vartheta(dv)ds)\\
&:=& J^1_{\varepsilon}+J^2_{\varepsilon}+J^3_{\varepsilon}+J^4_{\varepsilon}+J^5_{\varepsilon}+J^6_{\varepsilon},\\
\tilde{\theta}^{\varepsilon}(t)&=& \theta_0-\int^t_0 A_2 \tilde{\theta}^{\varepsilon}(s)ds-\int^t_0 \tilde{B}(\tilde{u}^{\varepsilon}(s),\tilde{\theta}^{\varepsilon}(s))ds -\int^t_0 f(\tilde{\theta}^{\varepsilon}(s))ds\\
&:=& K^1_{\varepsilon}+K^2_{\varepsilon}+K^3_{\varepsilon}+K^4_{\varepsilon}.
\end{eqnarray*}
Clearly,
\begin{eqnarray*}
&&\sup_{0<\varepsilon<\varepsilon_0}\mathbb{E}|J^1_{\varepsilon}|^2_{\mathbb{H}}\leq L_1, \quad \sup_{0<\varepsilon<\varepsilon_0}\mathbb{E}|K^1_{\varepsilon}|^2_{H^1}\leq R_1.
\end{eqnarray*}
By the same method as in the proof of (\ref{eq-50}) and (\ref{eq-51}), we get
\begin{eqnarray*}
\sup_{0<\varepsilon<\varepsilon_0}\mathbb{E}\|J^2_{\varepsilon}\|^2_{W^{\alpha,2}([0,T];\mathbb{V}')}&\leq & CT\sup_{0<\varepsilon<\varepsilon_0}\mathbb{E} \int^T_0\|\tilde{u}^{\varepsilon}(r)\|^2dr\leq L_2(\alpha),\\
\sup_{0<\varepsilon<\varepsilon_0}\mathbb{E}\|K^2_{\varepsilon}\|^2_{W^{\alpha,2}([0,T];(H^2)')}&\leq & CT\sup_{0<\varepsilon<\varepsilon_0}\mathbb{E} \sup_{t\in[0,T]}|\tilde{\theta}^{\varepsilon}(t)|^2\leq R_2(\alpha),
\end{eqnarray*}
and
\begin{eqnarray*}
\sup_{0<\varepsilon<\varepsilon_0}\mathbb{E}\|J^3_{\varepsilon}\|_{W^{\alpha,2}([0,T];\mathbb{V}')}&\leq & CT\sup_{0<\varepsilon<\varepsilon_0}\Big(\mathbb{E}\sup_{t\in[0,T]}|\tilde{u}^{\varepsilon}(t)|^2\Big)^{\frac{1}{2}}\Big(\mathbb{E}\int^t_s\|\tilde{u}^{\varepsilon}(r)\|^2dr\Big)^{\frac{1}{2}}\\
&\leq& L_3(\alpha),\\
\sup_{0<\varepsilon<\varepsilon_0}\mathbb{E}\|K^3_{\varepsilon}\|_{W^{\alpha,2}([0,T];(H^2)')}&\leq &
CT^{\frac{1}{2}}\sup_{0<\varepsilon<\varepsilon_0}\Big(\mathbb{E}\sup_{t\in[0,T]}|\tilde{u}^{\varepsilon}(t)|^2\Big)^{\frac{1}{2}}
\Big(\mathbb{E}\int^t_s\|\tilde{u}^{\varepsilon}(r)\|^2dr\Big)^{\frac{1}{2}}
\\
&&\ +CT^{\frac{1}{2}}\sup_{0<\varepsilon<\varepsilon_0}\Big(\mathbb{E}\sup_{t\in[0,T]}\|\tilde{\theta}^{\varepsilon}(t)\|^2\Big)^{\frac{1}{2}}\Big(\mathbb{E}\int^t_s|\Delta \tilde{\theta}^{\varepsilon}(r)|^2dr\Big)^{\frac{1}{2}}\\
&\leq& R_3(\alpha).
\end{eqnarray*}
Moreover, we get
\begin{eqnarray*}
\sup_{0<\varepsilon<\varepsilon_0}\mathbb{E}\|J^4_{\varepsilon}\|_{W^{\alpha,2}([0,T];\mathbb{V}')}&\leq &
CT^{\frac{1}{2}}\sup_{0<\varepsilon<\varepsilon_0}\Big(\mathbb{E}\sup_{t\in[0,T]}\| \tilde{\theta}^{\varepsilon}(t)\|^2\Big)^{\frac{1}{2}}\Big(\mathbb{E}\int^t_s|\Delta \tilde{\theta}^{\varepsilon}(r)|^2dr\Big)^{\frac{1}{2}}\leq L_4(\alpha),\\
\sup_{0<\varepsilon<\varepsilon_0}\mathbb{E}\|K^4_{\varepsilon}\|^2_{W^{\alpha,2}([0,T];(H^2)')}&\leq &
CT^2+CT^2\sup_{0<\varepsilon<\varepsilon_0}\mathbb{E}\sup_{t\in [0,T]}\|\tilde{\theta}^{\varepsilon}(t)\|^{4N+2}_{H^1}\leq R_4(\alpha),
\end{eqnarray*}
where (\ref{eq-61})-(\ref{eq-62}) are used.

For the remain two terms $J^5_{\varepsilon}$ and $J^6_{\varepsilon}$, referring to Lemma 4.2 in \cite{Z-Z}, we have
\begin{eqnarray*}
\sup_{0<\varepsilon<\varepsilon_0}\mathbb{E}\|J^5_{\varepsilon}\|^2_{W^{1,2}([0,T];\mathbb{H})}\leq L_5,\quad \sup_{0<\varepsilon<\varepsilon_0}\mathbb{E}\|J^6_{\varepsilon}\|^2_{W^{1,2}([0,T];\mathbb{H})}\leq L_6.
\end{eqnarray*}
Based on all the above estimates, we complete the proof.

\end{proof}

To prove (ii) in \textbf{Condition A}, we need to obtain the tightness of $\{(\tilde{u}^{\varepsilon},\tilde{\theta}^{\varepsilon})\}_{0<\varepsilon<\varepsilon_0}$ in $\mathcal{D}([0,T];D(A^{-\alpha}_1))\times C([0,T];H^{-2})$, for some $\alpha>1$.

Recall the following two lemmas related to the tightness of $\{(\tilde{u}^{\varepsilon},\tilde{\theta}^{\varepsilon})\}_{0<\varepsilon<\varepsilon_0}$. The proof can be found in \cite{J} and \cite{A-1}.
\begin{lemma}
Let $E$ be a separable Hilbert space with the inner product $(\cdot,\cdot)$. For an orthonormal basis $\{\xi_k\}_{k\in \mathbb{N}}$ in $E$, define the function $r^2_{N}:E\rightarrow \mathbb{R}^{+}$ by
\[
r^2_{N}(x)=\sum_{k\geq N+1}(x,\xi_k)^2,\quad N\in \mathbb{N}.
\]
Let $E_0$ be a total and closed under addition subset of $E$. Then a sequence $\{X_{\varepsilon}\}_{\varepsilon\in(0,1)}$ of stochastic process with trajectories in $\mathcal{D}([0,T],E)$ iff the following \textbf{Condition B } holds:
 \begin{description}
  \item[1.] $\{X_{\varepsilon}\}_{\varepsilon\in(0,1)}$ is $E_0-$weakly tight, that is, for every $h\in E_0$, $\{(X_{\varepsilon},h)\}_{\varepsilon\in(0,1)}$ is tight in $\mathcal{D}([0,T];\mathbb{R})$,
  \item[2.] For every $\eta>0$,
  \begin{eqnarray}\label{eq-65}
\lim_{N\rightarrow \infty}\lim_{\varepsilon\rightarrow 0}P\Big(r^2_{N}(X_{\varepsilon}(s)>\eta)\ for\ some\ s\in[0,T]\Big)=0.
\end{eqnarray}
\end{description}

\end{lemma}

Consider a sequence $\{\tau_{\varepsilon},\delta_{\varepsilon}\}$ satisfying the following \textbf{Condition C}:
\begin{description}
  \item[(1)] For each $\varepsilon$, $\tau_{\varepsilon}$ ia a stopping time with respect to the natural $\sigma-$fildes, and takes only finitely many values.
  \item[(2)] The constant $\delta_{\varepsilon}\in[0,T]$ satisfying $\delta_{\varepsilon}\rightarrow 0$ as $\varepsilon\rightarrow 0$.
\end{description}
Let $\{Y_{\varepsilon}\}_{\varepsilon\in(0,1)}$ be a sequence of random elements of $\mathcal{D}([0,T];\mathbb{R})$. For $f\in \mathcal{D}([0,T];\mathbb{R})$,  let $J(f)$ denote the maximum of the jump $|f(t)-f(t-)|$. We introduce the following \textbf{Condition D } on $\{Y_{\varepsilon}\}$:
\begin{description}
  \item[{(I)}] For each sequence $\{\tau_{\varepsilon},\delta_{\varepsilon}\}$ satisfying \textbf{Condition C}, $Y_{\varepsilon}(\tau_\varepsilon+\delta_\varepsilon)-Y_{\varepsilon}(\tau_\varepsilon)\rightarrow 0$ in probability, as $\varepsilon\rightarrow 0$.
\end{description}
\begin{lemma}\label{lem-11}
Assume $\{Y_{\varepsilon}\}_{\varepsilon\in(0,1)}$ satisfies \textbf{Condition D}, and either $\{Y_{\varepsilon}(0)\}$ and $J(Y_{\varepsilon})$ are tight on the line or $\{Y_{\varepsilon}(t)\}$ is tight on the line for each $t\in[0,T]$, then $\{Y_{\varepsilon}\}$ is tight in $\mathcal{D}([0,T];\mathbb{R})$.
\end{lemma}

Let $(\tilde{u}^{\varepsilon},\tilde{\theta}^{\varepsilon})$ be defined by (\ref{eq-63}). We have
\begin{lemma}\label{lem-13}
$\{(\tilde{u}^{\varepsilon},\tilde{\theta}^{\varepsilon})\}_{0<\varepsilon<\varepsilon_0}$ is tight in  $\mathcal{D}([0,T];D(A^{-\alpha}_1))\times C([0,T];H^{-2})$, for some $\alpha>1$.
\end{lemma}
\begin{proof}
With the help of (\ref{eq-60}) and $\tilde{\theta}^{\varepsilon}\in C([0,T];H^1)$, we deduce that $\tilde{\theta}^{\varepsilon}$ is tight in $C([0,T];H^{-2})$. Now, we prove $\{\tilde{u}^{\varepsilon}\}_{0<\varepsilon<\varepsilon_0}$ is tight in $\mathcal{D}([0,T];D(A^{-\alpha}_1))$.
Note that $\{\lambda^{\alpha}_i \varrho_i\}_{i\in \mathbb{N}}$ is a complete orthonormal system of $D(A^{-\alpha}_1)$. Since
\begin{eqnarray*}
\lim_{N\rightarrow \infty}\lim_{\varepsilon\rightarrow 0}\mathbb{E}\sup_{t\in[0,T]}r^2_{N}(\tilde{u}^{\varepsilon}(s))&=&\lim_{N\rightarrow \infty}\lim_{\varepsilon\rightarrow 0}\mathbb{E}\sup_{t\in[0,T]}\sum^{\infty}_{i=N+1}(\tilde{u}^{\varepsilon}(s),\lambda^{\alpha}_i \varrho_i)^2_{D(A^{-\alpha}_1)}\\
&=&\lim_{N\rightarrow \infty}\lim_{\varepsilon\rightarrow 0}\mathbb{E}\sup_{t\in[0,T]}\sum^{\infty}_{i=N+1}(A^{-\alpha}_1\tilde{u}^{\varepsilon}(s), \varrho_i)^2_{\mathbb{H}}\\
&=&\lim_{N\rightarrow \infty}\lim_{\varepsilon\rightarrow 0}\mathbb{E}\sup_{t\in[0,T]}\sum^{\infty}_{i=N+1}\frac{(\tilde{u}^{\varepsilon}(s), \varrho_i)^2_{\mathbb{H}}}{\lambda^{2\alpha}_i}\\
&\leq&\lim_{N\rightarrow \infty}\frac{\lim_{\varepsilon\rightarrow 0}\mathbb{E}\sup_{t\in[0,T]}|\tilde{u}^{\varepsilon}(t)|^2_{\mathbb{H}}}{\lambda^{2\alpha}_{N+1}}\\
&=&0,
\end{eqnarray*}
which implies (\ref{eq-65}) holds with $E=D(A^{-\alpha}_1)$.

Choosing $E_0=D(A^{\alpha}_1)$. We now prove $\{\tilde{u}^{\varepsilon}, 0<\varepsilon <\varepsilon_0\}$ is $E_0-$weakly tight. Let $h\in D(A^{\alpha}_1)$, and $\{\tau_{\varepsilon},\delta_{\varepsilon}\}$ satisfies
\textbf{Condition C}. It's easy to see $\{(\tilde{u}^{\varepsilon}(t), h)_{E}, 0<\varepsilon<\varepsilon_0\}$ is tight on the real line for each $t\in[0,T]$.

We now prove that $\{(\tilde{u}^{\varepsilon}(t), h)_{E}, 0<\varepsilon<\varepsilon_0\}$ satisfies \textbf{(D)}.
From (\ref{eq-67})-(\ref{eq-66}), we have
\begin{eqnarray*}
\tilde{u}^{\varepsilon}(\tau_{\varepsilon}+\delta_{\varepsilon})-\tilde{u}^{\varepsilon}(\tau_{\varepsilon})&=& -\int^{\tau_{\varepsilon}+\delta_{\varepsilon}}_{\tau_{\varepsilon}} A_1 \tilde{u}^{\varepsilon}(s)ds-\int^{\tau_{\varepsilon}+\delta_{\varepsilon}}_{\tau_{\varepsilon}} B(\tilde{u}^{\varepsilon}(s))ds\\
&& \ -\int^{\tau_{\varepsilon}+\delta_{\varepsilon}}_{\tau_{\varepsilon}} M(\tilde{\theta}^{\varepsilon}(s))ds+\int^{\tau_{\varepsilon}+\delta_{\varepsilon}}_{\tau_{\varepsilon}} \int_{\mathbb{X}}G(s,\tilde{u}^{\varepsilon}(s),v)(\varphi_{\varepsilon}(s,v)-1)\vartheta(dv)ds)\\
&&\ +\int^{\tau_{\varepsilon}+\delta_{\varepsilon}}_{\tau_{\varepsilon}} \int_{\mathbb{X}}\varepsilon G(s,\tilde{u}^{\varepsilon}(s),v)( N^{\varepsilon^{-1}\varphi_{\varepsilon}}(dsdv)-\varepsilon^{-1}\varphi_{\varepsilon}(s,v)\vartheta(dv)ds)\\
&:=& K^\varepsilon_1+K^\varepsilon_2+K^\varepsilon_3+K^\varepsilon_4+K^\varepsilon_5.
\end{eqnarray*}
Clearly, $\lim_{\varepsilon\rightarrow0}\mathbb{E}|(K^\varepsilon_5,h)_{E}|^2=0$.
For $K^\varepsilon_1$, using (\ref{eq-62}), we have
\begin{eqnarray*}
\lim_{\varepsilon\rightarrow0}\mathbb{E}|(K^\varepsilon_1,h)_{E}|\leq \|h\|_{D(A_1)}\lim_{\varepsilon\rightarrow0}\delta_\varepsilon \mathbb{E}[\sup_{t\in[0,T]}|\tilde{u}^{\varepsilon}(t)|_{\mathbb{H}}]=0.
\end{eqnarray*}
By (\ref{eq-3}), we get
\begin{eqnarray*}
\lim_{\varepsilon\rightarrow0}\mathbb{E}|(K^\varepsilon_2,h)_{E}|
&\leq&\|h\|_{\mathbb{V}}\lim_{\varepsilon\rightarrow0}\mathbb{E}\int^{\tau_{\varepsilon}+\delta_{\varepsilon}}_{\tau_{\varepsilon}}
\|B(\tilde{u}^{\varepsilon}(t))\|_{\mathbb{V}'}dt\\
&\leq&\|h\|_{\mathbb{V}}\lim_{\varepsilon\rightarrow0}\mathbb{E}\int^{\tau_{\varepsilon}+\delta_{\varepsilon}}_{\tau_{\varepsilon}}
|\tilde{u}^{\varepsilon}(t)|^2dt\\
&\leq&\|h\|_{\mathbb{V}}\lim_{\varepsilon\rightarrow0}\delta_\varepsilon \mathbb{E}\sup_{t\in[0,T]}
|\tilde{u}^{\varepsilon}(t)|^2\\
&=&0.
\end{eqnarray*}
With the help of (\ref{eq-4}) and (\ref{eq-62}), we deduce that
\begin{eqnarray*}
\lim_{\varepsilon\rightarrow0}\mathbb{E}|(K^\varepsilon_3,h)_{E}|
&\leq&\|h\|_{\mathbb{V}}\lim_{\varepsilon\rightarrow0}\mathbb{E}\int^{\tau_{\varepsilon}+\delta_{\varepsilon}}_{\tau_{\varepsilon}}
\|M(\tilde{\theta}^{\varepsilon}(t))\|_{\mathbb{V}'}dt\\
&\leq&\|h\|_{\mathbb{V}}\lim_{\varepsilon\rightarrow0}\mathbb{E}\int^{\tau_{\varepsilon}+\delta_{\varepsilon}}_{\tau_{\varepsilon}}
\|\tilde{\theta}^{\varepsilon}(t)\||\Delta\tilde{\theta}^{\varepsilon}(t)|dt\\
&\leq&\|h\|_{\mathbb{V}}\lim_{\varepsilon\rightarrow0}\mathbb{E}[\sup_{t\in[0,T]}\|\tilde{\theta}^{\varepsilon}(t)\|\int^{\tau_{\varepsilon}+\delta_{\varepsilon}}_{\tau_{\varepsilon}}
|\Delta\tilde{\theta}^{\varepsilon}(t)|dt]\\
&\leq&\|h\|_{\mathbb{V}}\lim_{\varepsilon\rightarrow0}\delta^{\frac{1}{2}}_{\varepsilon}\Big(\mathbb{E}\sup_{t\in[0,T]}\|\tilde{\theta}^{\varepsilon}(t)\|^2\Big)^{\frac{1}{2}}
\Big(\mathbb{E}\int^{\tau_{\varepsilon}+\delta_{\varepsilon}}_{\tau_{\varepsilon}}
|\Delta\tilde{\theta}^{\varepsilon}(t)|^2dt\Big)^{\frac{1}{2}}\\
&=&0.
\end{eqnarray*}
For $K^\varepsilon_4$, referring to (4.82) in \cite{Z-Z}, we get
\begin{eqnarray*}
\lim_{\varepsilon\rightarrow0}\mathbb{E}|(K^\varepsilon_4,h)_{E}|=0.
\end{eqnarray*}
Hence, we conclude the desired result.

\end{proof}

Fix the solution $(\tilde{u}^{\varepsilon},\tilde{\theta}^{\varepsilon})$ of (\ref{eq-67})-(\ref{eq-66}), consider the following equation:
\begin{eqnarray}\label{eq-68}
d\tilde{\xi}^{\varepsilon}(t)=-A_1\tilde{\xi}^{\varepsilon}(t)dt
+\varepsilon\int_{\mathbb{X}}G(t,\tilde{u}^{\varepsilon}(t-),v)
(N^{\varepsilon^{-1}\varphi_\varepsilon}(dtdv)-\varepsilon^{-1}\varphi_\varepsilon(t,v)\vartheta(dv)dt),
\end{eqnarray}
with $\tilde{\xi}^{\varepsilon}(0)=0$. Referring to Proposition 3.1 in \cite{R-Z}, there exists a unique solution $\tilde{\xi}^{\varepsilon}(t), t\geq 0$
to (\ref{eq-68}). Moreover,
\begin{eqnarray}\label{eq-69}
\tilde{\xi}^{\varepsilon}\in \mathcal{D}([0,T];\mathbb{H})\cap L^2([0,T];\mathbb{V}),
\end{eqnarray}
and there exists constant $C$ and $\tilde{\varepsilon}_0<\varepsilon_0$ such that for any $0<\varepsilon<\tilde{\varepsilon}_0$,
\begin{eqnarray}\label{eq-70}
\mathbb{E}\sup_{t\in[0,T]}|\tilde{\xi}^{\varepsilon}|^2_{\mathbb{H}}+\mathbb{E}\int^T_0\|\tilde{\xi}^{\varepsilon}\|^2_{\mathbb{V}}dt
\leq \sqrt{\varepsilon}C.
\end{eqnarray}

Now, we are ready to prove (ii) in \textbf{Condition A}. Recall $\mathcal{G}^{\varepsilon}(\varepsilon N^{\varepsilon^{-1}\varphi_{\varepsilon}})=(\tilde{u}^{\varepsilon},\tilde{\theta}^{\varepsilon})$ is defined by (\ref{eq-63}).

\begin{thm}\label{thm-3}
Fix $M\in \mathbb{N}$, and let $\{\varphi_\varepsilon,\varepsilon<\varepsilon_0\}\subset \mathcal{U}^M, \varphi\in\mathcal{U}^M$ be such that $\varphi_\varepsilon$ converges in distribution to $\varphi$ as $\varepsilon\rightarrow 0$. Then
\begin{eqnarray*}
\mathcal{G}^{\varepsilon}(\varepsilon N^{\varepsilon^{-1}\varphi_\varepsilon})\ converges\ in\ distribution\ to\  \mathcal{G}^{0}(\vartheta^\varphi_T),
\end{eqnarray*}
in $\mathcal{D}([0,T];\mathbb{H})\times C([0,T];H^1)$.
\end{thm}
\begin{proof}
Note that $\mathcal{G}^{\varepsilon}(\varepsilon N^{\varepsilon^{-1}\varphi_{\varepsilon}})=(\tilde{u}^{\varepsilon},\tilde{\theta}^{\varepsilon})$. From Lemma \ref{lem-10} and Lemma \ref{lem-13}, we know that
\begin{description}
  \item[1.] $\tilde{u}^{\varepsilon}$ is tight in $\mathcal{D}([0,T];D(A^{-\alpha}_1))\cap L^2([0,T];\mathbb{H})$, \ for\ $\alpha>1$,
  \item[2.] $\tilde{\theta}^{\varepsilon}$ is tight in $C([0,T];(H^2)')\cap L^2([0,T];H^1)$,
  \item[3.]$\lim_{\varepsilon\rightarrow 0}\mathbb{E}\Big[\sup_{t\in[0,T]}|\tilde{\xi}^{\varepsilon}(t)|^2_{\mathbb{H}}+\int^T_0\|\tilde{\xi}^{\varepsilon}(t)\|^2_{\mathbb{V}}dt\Big]
=0,$
\end{description}
where $\tilde{\xi}^{\varepsilon}$ is defined in (\ref{eq-68}). Let $\Xi=\Big(\mathcal{D}([0,T];D(A^{-\alpha}_1))\cap L^2([0,T];\mathbb{H})\Big)\times \Big(C([0,T];(H^2)')\cap L^2([0,T];H^1)\Big)$. Set
\[
\Pi=(\Xi, \mathcal{U}^M, \mathcal{D}([0,T];\mathbb{H})\cap L^2([0,T];\mathbb{V})).
\]
Let $((\tilde{u},\tilde{\theta}), \varphi, 0)$ be any limit of the tight family $\{((\tilde{u}^{\varepsilon},\tilde{\theta}^{\varepsilon}), \varphi_\varepsilon, \tilde{\xi}^{\varepsilon}),\varepsilon\in (0,\tilde{\varepsilon}_0)\}$. We will show that $(\tilde{u},\tilde{\theta})$
has the same law as $\mathcal{G}^{0}(\vartheta^\varphi_T)$ and $(\tilde{u}^\varepsilon,\tilde{\theta}^\varepsilon)$ converges in distribution to $(\tilde{u},\tilde{\theta})$ in $\mathcal{D}([0,T];\mathbb{H})\times C([0,T];H^1)$.

By the Skorokhod representative theorem, there exists a stochastic basis $(\Omega^1, \mathcal{F}^1, \{\mathcal{F}^1_t\}_{t\in [0,T]}, \mathbb{P}^1)$ and, on this basis, $\Pi-$valued random variables
$((\tilde{u}_1,\tilde{\theta}_1),\varphi^1,0), ((\tilde{u}^\varepsilon_1,\tilde{\theta}^\varepsilon_1), \varphi^1_\varepsilon,\tilde{\xi}^{\varepsilon}_1)$ such that $((\tilde{u}^\varepsilon_1,\tilde{\theta}^\varepsilon_1), \varphi^1_\varepsilon,\tilde{\xi}^{\varepsilon}_1)$ (resp. $((\tilde{u}_1,\tilde{\theta}_1),\varphi^1,0)$) has the same law as $((\tilde{u}^{\varepsilon},\tilde{\theta}^{\varepsilon}), \varphi_\varepsilon, \tilde{\xi}^{\varepsilon})$ (resp. $((\tilde{u},\tilde{\theta}), \varphi, 0)$), and $((\tilde{u}^\varepsilon_1,\tilde{\theta}^\varepsilon_1), \varphi^1_\varepsilon,\tilde{\xi}^{\varepsilon}_1)\rightarrow ((\tilde{u}_1,\tilde{\theta}_1),\varphi^1,0)$ in $\Pi$, $\mathbb{P}^1-$a.s.

From the equations satisfied by $((\tilde{u}^{\varepsilon},\tilde{\theta}^{\varepsilon}), \varphi_\varepsilon, \tilde{\xi}^{\varepsilon})$, we see that $((\tilde{u}^\varepsilon_1,\tilde{\theta}^\varepsilon_1), \varphi^1_\varepsilon,\tilde{\xi}^{\varepsilon}_1)$ satisfies the following integral equations:
\begin{eqnarray*}
\tilde{u}^\varepsilon_1(t)-\tilde{\xi}^{\varepsilon}_1(t)
&=&u_0-\int^t_0A_1(\tilde{u}^\varepsilon_1(s)-\tilde{\xi}^{\varepsilon}_1(s))ds
-\int^t_0B(\tilde{u}^\varepsilon_1(s))ds\\
&&\ -\int^t_0M(\tilde{u}^\varepsilon_1(s),\tilde{\theta}^\varepsilon_1(s))ds
+\int^t_0\int_{\mathbb{X}}G(s,\tilde{u}^\varepsilon_1(s),v)(\varphi^1_\varepsilon(s,v)-1)\vartheta(dv)ds,\\
\tilde{\theta}^\varepsilon_1(t)&=&\theta_0-\int^t_0A_2 \tilde{\theta}^\varepsilon_1(s)ds-\int^t_0 \tilde{B}(\tilde{u}^\varepsilon_1(s),\tilde{\theta}^\varepsilon_1(s))ds-\int^t_0 f(\tilde{\theta}^\varepsilon_1(s))ds.
\end{eqnarray*}
Define $\Sigma=(C([0,T];\mathbb{H})\cap L^2([0,T];\mathbb{V}))\times (C([0,T];H^1)\cap L^2([0,T];H^2))$, we have
\begin{eqnarray*}
\mathbb{P}^1\Big((\tilde{u}^\varepsilon_1-\tilde{\xi}^{\varepsilon}_1,\tilde{\theta}^\varepsilon_1)\in \Sigma\Big)=\bar{\mathbb{P}}\Big((\tilde{u}^\varepsilon-\tilde{\xi}^{\varepsilon},\tilde{\theta}^\varepsilon)\in \Sigma\Big)=1.
\end{eqnarray*}
Let $\Omega^1_0$ be the subset of $\Omega^1$ such that $((\tilde{u}^\varepsilon_1,\tilde{\theta}^\varepsilon_1), \varphi^1_\varepsilon,\tilde{\xi}^{\varepsilon}_1)\rightarrow ((\tilde{u}_1,\tilde{\theta}_1),\varphi^1,0)$ in $\Pi$, then $\mathbb{P}^1(\Omega^1_0)=1$. Now, we have to show that, for any fixed $\omega^1\in\Omega^1_0$,
\begin{eqnarray}\label{eq-71}
\sup_{t\in[0,T]}|\tilde{u}^\varepsilon_1(\omega^1,t)-\tilde{u}_1(\omega^1,t)|^2_{\mathbb{H}}\rightarrow 0\quad and\
\sup_{t\in[0,T]}\|\tilde{\theta}^\varepsilon_1(\omega^1,t)-\tilde{\theta}_1(\omega^1,t)\|^2_{H^1}\rightarrow 0, \ as \ \varepsilon\rightarrow 0.
\end{eqnarray}
Set $p^\varepsilon(t)=\tilde{u}^\varepsilon_1(t)-\tilde{\xi}^{\varepsilon}_1(t)$ and $q^\varepsilon(t)=\tilde{\theta}^\varepsilon_1(t)$. Then, $(p^\varepsilon(\omega^1,t),q^\varepsilon(\omega^1,t))$
satisfies
\begin{eqnarray*}
p^\varepsilon(t)&=&u_0-\int^t_0A_1p^\varepsilon(s)ds
-\int^t_0B(p^\varepsilon(s)+\tilde{\xi}^{\varepsilon}_1(s))ds\\
&&\ -\int^t_0M(p^\varepsilon(s)+\tilde{\xi}^{\varepsilon}_1(s),q^\varepsilon(s))ds
+\int^t_0\int_{\mathbb{X}}G(s,p^\varepsilon(s)+\tilde{\xi}^{\varepsilon}_1(s),v)(\varphi^1_\varepsilon(s,v)-1)\vartheta(dv)ds,\\
q^\varepsilon(t)&=&\theta_0-\int^t_0A_2 q^\varepsilon(s)ds-\int^t_0 \tilde{B}(p^\varepsilon(s)+\tilde{\xi}^{\varepsilon}_1(s),q^\varepsilon(s))ds-\int^t_0 f(q^\varepsilon(s))ds.
\end{eqnarray*}

Since
\[
\lim_{\varepsilon\rightarrow 0}[\sup_{t\in[0,T]}|\tilde{\xi}^{\varepsilon}(\omega^1,t)|^2_{\mathbb{H}}+\int^T_0\|\tilde{\xi}^{\varepsilon}(\omega^1,t)\|^2_{\mathbb{V}}dt]
=0,
\]
we have
\begin{eqnarray}\notag
&&\lim_{\varepsilon\rightarrow 0}\sup_{t\in[0,T]}|\tilde{u}^{\varepsilon}_1(\omega^1,t)-\hat{u}(\omega^1,t)|^2_{\mathbb{H}}+\lim_{\varepsilon\rightarrow 0}\sup_{t\in[0,T]}\|\tilde{\theta}^\varepsilon_1(t)-\hat{\theta}(\omega^1,t)\|^2_{H^1}\\ \notag
&\leq& \lim_{\varepsilon\rightarrow 0}\sup_{t\in[0,T]}[|p^{\varepsilon}(\omega^1,t)-\hat{u}(\omega^1,t)|^2_{\mathbb{H}}+|\tilde{\xi}^{\varepsilon}_1(\omega^1,t)|^2_{\mathbb{H}}]
 +\lim_{\varepsilon\rightarrow 0}\sup_{t\in[0,T]}\|q^{\varepsilon}(\omega^1,t)-\hat{\theta}(\omega^1,t)\|^2_{H^1}\\ \label{eq-73}
&=&\lim_{\varepsilon\rightarrow 0}\sup_{t\in[0,T]}|p^{\varepsilon}(\omega^1,t)-\hat{u}(\omega^1,t)|^2_{\mathbb{H}}+\lim_{\varepsilon\rightarrow 0}\sup_{t\in[0,T]}\|q^{\varepsilon}(\omega^1,t)-\hat{\theta}(\omega^1,t)\|^2_{H^1}.
\end{eqnarray}
By the similar method as Theorem \ref{thm-2}, we can obtain that
\begin{eqnarray}\label{eq-72}
\lim_{\varepsilon\rightarrow 0}\sup_{t\in[0,T]}|p^{\varepsilon}(\omega^1,t)-\hat{u}(\omega^1,t)|^2_{\mathbb{H}}=0, \quad {\rm{and}} \quad \lim_{\varepsilon\rightarrow 0}\sup_{t\in[0,T]}\|q^{\varepsilon}(\omega^1,t)-\hat{\theta}(\omega^1,t)\|^2_{H^1}=0,
\end{eqnarray}
where
\begin{eqnarray*}
\hat{u}(t)&=&u_0-\int^t_0A_1\hat{u}(s)ds-\int^t_0B(\hat{u}(s))ds-\int^t_0M(\hat{\theta}(s))ds\\
&&\ +\int^t_0\int_{\mathbb{X}}G(s,\hat{u}(s),v)(\varphi^1(s,v)-1)\vartheta(dv)ds,\\
\hat{\theta}(t)&=&\theta_0-\int^t_0A_2\hat{\theta}(s)ds-\int^t_0\tilde{B}(\hat{u}(s),\hat{\theta}(s))ds-\int^t_0f(\hat{\theta}(s))ds.
\end{eqnarray*}
Hence, combining (\ref{eq-73}) and (\ref{eq-72}), we obtain
 \begin{eqnarray}\label{eq-72-1}
\lim_{\varepsilon\rightarrow 0}\sup_{t\in[0,T]}|\tilde{u}^{\varepsilon}_1(\omega^1,t)-\hat{u}(\omega^1,t)|^2_{\mathbb{H}}=0,
\quad {\rm{and}} \quad \lim_{\varepsilon\rightarrow 0}\sup_{t\in[0,T]}\|\tilde{\theta}^\varepsilon_1(t)-\hat{\theta}(\omega^1,t)\|^2_{H^1}=0,
\end{eqnarray}
which imply that
 $(\tilde{u}_1, \tilde{\theta}_1)=(\hat{u},\hat{\theta})=\mathcal{G}^0(\vartheta^{\varphi^1})$, and $(\tilde{u}, \tilde{\theta})$ has the same law as $\mathcal{G}^0(\vartheta^{\varphi})$.
Since $(\tilde{u}^{\varepsilon}, \tilde{\theta}^{\varepsilon})=(\tilde{u}^{\varepsilon}_1, \tilde{\theta}^{\varepsilon}_1)$ in law, we deduce from (\ref{eq-72-1}) that $(\tilde{u}^{\varepsilon}, \tilde{\theta}^{\varepsilon})$ converges to $\mathcal{G}^0(\vartheta^{\varphi})$. We complete the proof.
\end{proof}


\vskip 0.2cm
\noindent{\bf  Acknowledgements}\
This work is supported by National Natural Science Foundation of China (Grant No. 11401057),  Natural Science Foundation Project of CQ  (Grant No. cstc2016jcyjA0326),
Fundamental Research Funds for the Central Universities(Grant No. 106112015CDJXY100005) and China Scholarship Council (Grant No.:201506055003).

\def\refname{ References}

\end{document}